\documentclass[11pt, oneside]{article}   	
\usepackage{geometry}                		
\usepackage{graphicx}				
\usepackage{epstopdf}
\usepackage[utf8]{inputenc}
\usepackage[english]{babel}
\usepackage[T1]{fontenc}
\usepackage{dsfont}
\usepackage{gensymb}
\usepackage{array}
\usepackage{latexsym}
\usepackage{amssymb,amsbsy,amsmath,amsfonts,amscd}
\usepackage{amsthm}
\usepackage{color,subfigure}
\usepackage{authblk}
\usepackage{ulem}
\newcommand{\commentout}[1]{}
\newcommand{\R}{\mathbb{R}}

\newcommand {\e}  {\varepsilon}

\newcommand {\Chi} {{\bf \raise 2pt \hbox{$\chi$}} }

\newcommand {\Ebar} {\overline{E}}
\newcommand {\Lbar} {\overline{L}}
\newcommand {\Abar} {\overline{A}}

\newcommand {\sgn} { {\rm sgn} }

\newcommand {\f}   {\frac}
\newcommand {\df}   {\displaystyle\frac}
\newcommand {\p}   {\partial}
\newcommand{\dis}{\displaystyle}
\newcommand{\tr}{\mathrm{tr}}
\newcommand{\beq}{\begin{equation}}
\newcommand{\eeq}{\end{equation}}
\newcommand{\bea} {\left\{ \begin{array}{rl}}
\newcommand{\eea} {\end{array}\right.}
\newcommand{\bepa}{\left\{ \begin{array}{l}}
\newcommand{\eepa} {\end{array}\right.}
\newcommand{\ueps}{u_{\e}}
\newcommand{\veps}{v_{\e}}
\newcommand{\weps}{w_{\e}}
\newcommand{\feps}{f_{\e}}
\newcommand{\geps}{g_{\e}}
\newcommand{\bartp}{\overline{\tau}_P}
\newcommand{\bardp}{\overline{\delta}_P}

\newtheorem{theorem}{Theorem}[section]
\newtheorem{lemma}[theorem]{Lemma}

\newtheorem{remark}[theorem]{Remark}

\newtheorem{proposition}[theorem]{Proposition}

\usepackage{array,multirow,makecell}
\setcellgapes{1pt}
\makegapedcells
\newcolumntype{R}[1]{>{\raggedleft\arraybackslash }b{#1}}
\newcolumntype{L}[1]{>{\raggedright\arraybackslash }b{#1}}
\newcolumntype{C}[1]{>{\centering\arraybackslash }b{#1}}

\usepackage[colorlinks=true]{hyperref}
\usepackage{cite}
\begin{document}

\title{Oscillatory regimes in a mosquito population model with larval feedback on egg hatching}

\author[1,2]{Martin Strugarek}
\author[2]{Laetitia Dufour}
\author[3]{Nicolas Vauchelet}
\author[2]{Luis Almeida}
\author[2]{Benoît Perthame}
\author[4]{Daniel A.M. Villela}
\affil[1]{AgroParisTech, 16 rue Claude Bernard, F-75231 Paris Cedex 05, France}
\affil[2]{Université Pierre et Marie Curie, Paris 6, Laboratoire Jacques–Louis Lions, UMR 7598 CNRS, 4 Place
Jussieu, 75252 Paris, France}
\affil[3]{LAGA - UMR 7539
 Institut Galil\'{e}e
 Universit\'{e} Paris 13
 99, avenue Jean-Baptiste Cl\'{e}ment
 93430 Villetaneuse - France}
\affil[4]{PROCC, Fundaç\~{a}o Oswaldo Cruz, Av. Brasil, 4365. Manguinhos
CEP 21040-360, Rio de Janeiro, Brazil}

\date{\today}
\maketitle

\begin{abstract}
Understanding mosquitoes life cycle is of great interest presently
because of the increasing impact of vector borne diseases in several countries.
There is evidence of oscillations in mosquito populations independent of seasonality, still unexplained, based on observations both in laboratories and in nature. We propose a simple mathematical model of egg hatching enhancement by larvae which
produces such oscillations that conveys a possible explanation. We propose both
a theoretical analysis, based on slow-fast dynamics and Hopf bifurcation, and
numerical investigations in order to shed some light on the mechanisms at work in this model.
\end{abstract}



\section*{Introduction}

Today numerous areas of the world are severely affected by mosquito-borne viral diseases, with notable examples including dengue, chikungunya and Zika (see \cite{Bha.Burden}). Scientists are hard at work to find new and efficient ways to mitigate the impact of, or even eradicate these arboviral infections, and especially target vector control.

A beneficial implementation of any of the vector control methods requires a good understanding of the local vector population's bio-ecology, and a reliable monitoring of its dynamics. To achieve better knowledge, this monitoring needs not only be demographic (using trap counts), but can also use genetic data - for the example of {\it Wolbachia} see \cite{Hof.Stability} and \cite{Yea.Mitochondrial}.
However, studies in Rio de Janeiro throughout the past decade have shown that monitoring urban populations of {\it Aedes aegypti} is a difficult task (see \cite{art:osc}, \cite{Dut.Lab}, \cite{Density}), largely because of environmental variations (spatial heterogeneity, seasonality, etc.).
A first - to the best of our knowledge - systematic comparison of two complex models of {\it Aedes aegypti} population dynamics, relevant for a control program, was done recently in \cite{Leg.Comparison}.
Another application of proper modeling of mosquito's life-cycle is the risk estimation for disease emergence (see \cite{Guz.Potential}).

We believe that the {\it intrinsic} life cycle of {\it Aedes aegpyti} may still be improperly modeled, and effort should be put in the direction of integrating several key features in the models. Among these features, we have in mind the {\it transitions} between the stages (egg, larva, pupa, adult) or even within these stages (larval instars, etc.) because in theory, any of these transitions (ovipositing behavior, hatching, pupation, mating, etc.) can give rise to nonlinearity. 
Nonlinearities ought to be taken into account when using collected data, so that they do not blur the picture we get of the actual population's dynamics.
In addition, synchronizing or de-synchronizing effects, either in time or space, are possible outputs of these nonlinearities, and can result in variations in crucial traits of the mosquito populations, such as vector capacity (see \cite{Jul.She}, \cite{Bar.Effect})

We focus exclusively in this work on one single aspect of the evolution of the mosquito population, setting the hypothesis that the larval density in breeding sites directly impacts the hatching rate. Previous works on hatching and larvae dynamics include \cite{Azn.Modeling}, \cite{Azn.Model}, where stochastic models with food dynamics were used. However, to the best of our knowledge, no mathematical work has been published on the very topic of hatching enhancement through larval density since the experimental findings of \cite{Liv.Complex}.
Observations on this phenomenon are uneasy to obtain in the field but can be assessed in the lab (see \cite{art:hatch}). Further research in this field could benefit from mathematical modeling tools able to take it into account and this may help monitoring the dynamics of mosquito populations.

We develop a mathematical model of the dynamics of mosquito population, with the requirements that this model be sufficiently generic to match experimental observations across various conditions and sufficiently simple so that it is possible to handle it theoretically and interpret it. Therefore we choose to develop a deterministic model based on a system of ordinary differential equations, as was done, for example, in \cite{AWD}.
Our simplistic model involves the positive influence of larvae on the system, acting on hatching rate. We show that this feature can explain oscillations.

We draw a general picture of the system's properties in Section \ref{sec:study}, and justify rigorously the use of a two-population model as a further simplification for the identification of the qualitative properties induced by hatching feedback.
Then we focus on two parameter regimes of particular interest.
Firstly (Section \ref{sec:slow-fast}) when the quantity of eggs is large compared to the quantity of larvae, oscillations can appear and we are faced to a slow-fast oscillatory regime giving rise to oscillation profiles comparable to those of the FitzHugh-Nagumo system (Theorem \ref{thm:mainslowfast}). We can compute the amplitude of the oscillations in this case, where they are typically large, and also their period. 
Secondly (Section \ref{lll}), we show that our model presents a Hopf bifurcation at any positive equilibrium of the system, assuming the quantity of larvae promotes hatching.
The bifurcation occurs as the feedback becomes stronger (Theorem \ref{mainres}). In this case we can compute the period of the oscillations at the bifurcation point.
We provide numerical results for the system parametrized (roughly) for a tropical area such as Rio de Janeiro, showing that the range of possible oscillations is wide.

\section{Models and their reduction}
\label{sec:reduction}
The life cycle of a mosquito (male and female) consists of two main stages: the aquatic stage (egg, larva, pupa), and the adult stage. 
We adopt a population biology point of view, which means that we describe the mosquitoes life-cycle thanks to a system of ordinary differential equations.
For the purpose of studying the impact of larval density on hatching, we introduce the number densities of each population: $A(t)$ (adults), $E(t)$ (eggs), $L(t)$ (larvae) and $P(t)$ (pupae).

In a compartmental model, one can suppose the following type of dynamics
\begin{equation}
  \left\{ \tag{$S_4$}
      \begin{aligned}
\frac{d}{dt} E &= \beta_E A - E \big( H(E,L) + \delta_E \big),  \\
\f{d}{dt} L &= E H(E,L) - L \big( \phi(L) + \delta_L + \tau_L \big),  \\
\f{d}{dt} P &= \tau_L L - \delta_P P - \tau_P P,  \\
\f{d}{dt} A &= \tau_P P - \delta_A A.
\end{aligned}
    \right.
    \label{eq:SG}
\end{equation}

We interpret the parameters as follows: $\beta_E>0$ is the intrinsic oviposition rate; $\delta_E, \delta_L, \delta_P, \delta_A >0$ are the death rates for eggs, larvae, pupae and adults, respectively; $\tau_L$, $\tau_P >0$ are the transition rates from larvae to pupae and pupae to adults, respectively ; $\phi$ tunes an extra-death term due to intra-specific competition (this term is non-linear and we assume that it depends only on the larval density);
finally, $H(E,L)$ is the hatching rate, which may in general depend on larval density~$L$ and on egg density~$E$, neglecting a possible effect of pupae.

In order to reduce \eqref{eq:SG} to a simpler model we suppose pupa population at equilibrium. This boils down to assuming that the time dynamics for pupae is fast compared to the other compartments and thus $P=\df{\tau_L}{\delta_P+\tau_P} L$.

To justify this approximation more rigorously, we assume $\tau_P, \delta_P = O(1/\e)$ (quantifying the ``fast dynamics'' for pupae) and define $\bartp = \e \tau_P$, $\bardp = \e \delta_P$.
We introduce $P = \e M$ and then we find the following equations on $M$ and $A$ (those on $E$ and $L$ are untouched)
\[
 \bepa
 \e \dis\f{d M}{dt} = \tau_L L - \e M (\tau_P + \delta_P),
 \\[10pt]
 \dis\f{d A}{d t} = \e \tau_P M - \delta_A A.
 \eepa
\]
This method follows the classical justification of Michaelis-Menten laws (see \cite{Murray}, \cite{Perthame}).
We end up with
\[
 \bepa
 \e \dis\f{d M}{dt} = \tau_L L -  M (\bartp+\bardp),
 \\[10pt]
 \dis\f{d A}{d t} = \bartp M - \delta_A A,
  \eepa
\]
and in the limit $\e \to 0$, we recover our claim under the form $M = \f{\tau_L}{\bartp+\bardp} L$.

This simplification enables us to reduce the model to dimension $3$.
From now on we also assume $H(E, L) = h(L)$ and $\phi(L) = c L$ in order to obtain the simplified system
\begin{equation}
  \left\{ \tag{$S_3$}
      \begin{aligned}
\frac{d}{dt} E &= \beta_E A - \delta_E E -h(L) E,  \\
\f{d}{dt} L &= h(L) E - \delta_L L - c L^2 - \tau_L L,  \\
\f{d}{dt} A &= \f{\tau_P\tau_L}{\delta_P+\tau_P} L - \delta_A A.
\end{aligned}
    \right.
    \label{eq:S3}
\end{equation}

We can proceed to a further reduction by supposing adult population at equilibrium. This boils down to assuming that the time dynamics for adult mosquitoes is fast compared to the other compartments.
Exactly as above with the pupae, in the approximation when $\delta_A$ and $\beta_E$ are large (and $A$ itself is small), it makes sense to set in this system, at first order, $A=\f{\tau_P\tau_L}{(\delta_P+\tau_P)\delta_A} L$.

Finally, system \eqref{eq:SG} reduces to the following system in dimension $2$:
\begin{equation}
  \left\{ 
      \begin{aligned}
       \frac{d}{dt} E&=b_E L - d_E E -h(L) E,\\
        \f{d}{dt}  L & =h(L) E - d_L L - c L^2,
      \end{aligned}
    \right.
    \label{eq:S}
\end{equation}
where $b_E=\beta_E\f{\tau_P\tau_L}{(\delta_P+\tau_P)\delta_A}>0$, $d_E=\delta_E>0$ and $d_L=\delta_L+\tau_L>0$.

We perform this model reduction because it is sufficient to take into account the larval effect.
Indeed, we show and quantify how the larval density-dependent hatching rate effectively generates oscillations, without any other source of instability (like time-delay, temperature variations or other environment-related effects).
However, for future practical applications, further studies including the use of a more biologically realistic model will be mandatory.

According to experimental data (results from \cite{art:hatch}) and mainly guided by a biological intuition 
we assume the hatching undergoes saturation for large values of $L$:
\beq
h \in \mathcal{C}^1 ([0, \infty)), \quad h > 0, \quad \max_L h(L) =: h_0 < +\infty.
\label{hyp:h}
\eeq
To ensure the instability of the trivial equilibrium $(0,0)$ and rule out population extinction, we assume
\begin{equation}
d_E d_L < h(0)(b_E-d_L).
\label{hyp:naturelle}
\end{equation}
We also assume, for the matter of simplification of later computations
\begin{equation}
b_E > d_L + d_E.
\label{hyp:params}
\end{equation}

For several mosquito species, it is actually possible to identify the biological parameters $\tau_L$, $\delta_A$, $\delta_E$, $\delta_L$ and the adult density at equilibrium on the field ({\it i.e} $A$ such that $\f{d}{dt} A =0$).
From the formula $L=A\frac{\delta_A}{\tau_L}$, we deduce larvae density at equilibrium on the field (this density is called $\Lbar$ throughout this paper).

We warn the reader about what we call ``equilibrium density on the field'' and about parameter values. We do not claim they can precisely reproduce population variations as observed in field experiments. We simply use rough estimation of their orders of magnitude so as to prove the concept of population oscillations due to density-dependent hatching rate. See paragraph \ref{discuss:nonlinearities} for additional comments.

This warning made, from now on we consider that parameters $b_E$, $d_E$, $d_L$,  adults and larvae density at equilibrium on the field are known; the competition parameter $c$ and the hatching function $h$ are unknown. 
The known parameters are set at a given place and temperature (see \cite{Density}, \cite{Temp}) and we work with a fixed temperature, so the previous biological parameters are fixed and time-independent.

Our general goal is thus to assert the possible range of remaining parameters $c$ and $h(L)$ depending on the qualitative properties of solutions.

\section{Study of the reduced model}
\label{sec:study}
\subsection{Basic properties, equilibria and their stability}

With the assumption \eqref{hyp:h} we know that solutions remain non-negative.
Furthermore, the trivial equilibrium $(0,0)$ is a steady state of \eqref{eq:S} and all the other steady states $(\Ebar,\Lbar)$ are determined by a non-linear relation on $\Lbar$
\begin{equation}
  \left\{ 
      \begin{aligned}
       \Ebar&= \frac{b_E \Lbar}{d_E+h(\Lbar)},\\
       c \Lbar & = b_E - d_L - \f{d_E b_E}{d_E + h(\Lbar)}.
      \end{aligned}
    \right.
    \label{eq:steady}
\end{equation}
We observe that solutions of \eqref{eq:steady} are positive if and only if $h(\Lbar) > \f{d_E d_L}{b_E-d_L}.$
In addition:
\begin{lemma} 
Assume \eqref{hyp:h} holds.
Then there is a constant $K > 0$ such that for all non-negative~$t$, $L(t) + E(t) \leq K$.
Moreover, there exists at least one positive steady state of \eqref{eq:S} if and only if
  \beq
 \min_{x \geq 0} \, \Big( c x + \f{d_E b_E}{d_E + h(x)} \Big) \leq b_E - d_L.
 \label{eq:sufcond}
 \eeq
Furthermore, all steady states $(\Ebar, \Lbar) \not= (0, 0)$ satisfy $
  0 < c \Lbar < b_E - d_L - \f{d_E b_E}{d_E + h_0}. $
\label{lem:SS}
\end{lemma}

For the first point, we do not use any property of $h$, but merely the fact that $c L^2 / L \to +\infty$ as $L \to +\infty$. Note that with estimates on $h$, more restrictive properties can be obtained, in the sense that one could construct strictly smaller positively stable and attractive sets.

\begin{proof}
We notice that
\[
	\f{d}{dt} \big( E + L \big) = b_E L - d_E E - d_L L - c L^2 \leq -d_E (E + L) + U_M,
\]
where $U_M := \f{( b_E + d_E - d_L)^2}{4 c}$ is the maximum of $L \mapsto (b_E + d_E - d_L)L - c L^2$.
Consequently the claim holds with $K = U_M / d_E$.

 Let 
 \[
  f(x) = c x + \f{d_E b_E}{d_E + h(x)} - (b_E - d_L).
 \]
 Then $\Lbar$ defines a steady state of \eqref{eq:S} if and only if $f(\Lbar) = 0$, by \eqref{eq:steady}.

 Continuity of $f$ yields the conclusion since $h_0 = \max h$.
 \end{proof}

From now on we always assume that \eqref{eq:sufcond} holds, so that there exists at least one positive steady state of \eqref{eq:S}.
Then we analyze the stability of those steady states. 
\begin{lemma}\label{condh0}
The steady state $(0, 0)$ is unstable (locally linearly) if and only if \eqref{hyp:naturelle} holds.

A non-trivial steady state $(\Ebar,\Lbar)$ of \eqref{eq:S} is unstable (locally linearly) if and only if
either
\label{lemmeinstable}  
\beq
h'(\Lbar)\Ebar-d_L-2c\Lbar-d_E-h(\Lbar) > 0,
\label{eq:unstable1}
\eeq
or 
\beq
c d_E \Lbar - d_E h'(\Lbar) \Ebar + c \Lbar h(\Lbar) < 0
\quad \text{ and } \quad
 h'(\Lbar)\Ebar-d_L-2c\Lbar-d_E-h(\Lbar) \leq 0.
 \label{eq:unstable2}
\eeq
\end{lemma}
\begin{proof}
We divide the proof into three steps.

Firstly we linearize system \eqref{eq:S} around a steady state $(\Ebar,\Lbar)$. Setting $E=\Ebar + e + \ldots$ and 
$L=\Lbar+\ell + \ldots$, we find
\begin{equation*}
  \left\{ 
      \begin{aligned}
       \frac{d}{dt} e&= b_E \ell - d_E e -h(\Lbar) e -h'(\Lbar) \Ebar \ell ,\\
        \f{d}{dt}  \ell & =h(\Lbar) e + h'(\Lbar)\Ebar \ell - d_L \ell - 2 c \Lbar \ell.
      \end{aligned}
    \right.
\end{equation*}
The eigenvalues $\lambda$ of the above linear system are given by the determinant
$$
\left| \begin{array}{cc}
-d_E - h(\Lbar) - \lambda   &  b_E - h'(\Lbar) \Ebar  \\[10pt]
h(\Lbar) & h'(\Lbar) \Ebar - d_L - 2c\Lbar - \lambda
\end{array}\right| = 0.
$$
After straightforward computations, we obtain:
\begin{equation}\label{eq:lambda}
\lambda^2 - \lambda \Big(h'(\Lbar)\Ebar - d_L -2c\Lbar - d_E - h(\Lbar)\Big)
+ d_E\big( d_L + 2c \Lbar -h'(\Lbar)\Ebar \big) + h(\Lbar) \big(d_L +2 c \Lbar -b_E\big) = 0.
\end{equation}

Secondly we look at the trivial steady-state.
Taking $\Ebar=\Lbar=0$ in equation \eqref{eq:lambda}, we obtain:
\begin{equation}
P(\lambda):=\lambda^2 + \lambda (d_L+d_E+h(0)) + d_Ed_L + h(0)(d_L-b_E) = 0.
\label{eq:lambda2}
\end{equation}

We are looking for the condition such that $(0, 0)$ is linearly unstable (we are interested in the conditions when the mosquito population does not tend to zero in nature). 
In other words, we expect that the polynomial $P$ has a root with positive real part. Since the first order coefficient is positive we end up with condition \eqref{hyp:naturelle} and the first point of the lemma is proved.

Finally we consider non-trivial steady states.
We rewrite \eqref{eq:lambda} as
\begin{equation*}
\lambda^2 - \tr(A)\lambda + \det(A) = 0,
\end{equation*}
where $A$ is the Jacobian matrix of the linearized system \eqref{eq:S}. Using \eqref{eq:steady} we find
\[
 d_E\big( d_L + 2c \Lbar -h'(\Lbar)\Ebar \big) + h(\Lbar) \big(d_L +2 c \Lbar -b_E\big) = 
 c d_E \Lbar - d_E h'(\Lbar) \Ebar + c \Lbar h(\Lbar),
\]
and thus
\beq
\bepa
\tr(A) = h'(\Lbar)\Ebar - d_L -2c\Lbar - d_E - h(\Lbar),
\\[10pt]
\det(A)  = c d_E \Lbar - d_E h'(\Lbar) \Ebar + c \Lbar h(\Lbar).
\eepa
\label{eq:trdet}
\eeq
The discriminant $\Delta$ of this polynomial is $\Delta =\big(\tr(A)\big)^2-4 \det(A)$ and the steady state is unstable if and only if there exists a root with positive real part.

There are two cases:
If $\Delta<0$ then the real part of the roots is $\frac{\tr(A)}{2}$.
The steady state is unstable if and only if $\tr(A)>0$.

If $\Delta\geq 0$ then the bigger root is $\frac{\tr(A) + \sqrt{\Delta}}{2}$.
Hence the steady state is unstable if and only if $\tr(A) > - \sqrt{\Delta}$.
This is true if and only if either $\tr(A) > 0$ or if $\det(A) < 0$ \text{ and } $\tr(A) \leq 0$.

\end{proof}

\begin{remark}
There is a link with the basic offspring number $Q_0$ (defined in \cite{DHM}).
This dimensionless number is the average number of offspring generated by a single fertilized mosquito:
from the method in \cite{VdDW}, we can compute $Q_0=\df{b_E h(0)}{d_L (d_E + h(0))}$. 

We remark that the first statement in Lemma~\ref{condh0} boils down to the classical property: trivial equilibrium point is unstable if and only if $Q_0 > 1$.
\end{remark}

\begin{remark}
As in nature we can observe oscillations of eggs and larvae density \cite{art:osc}, we pay attention in this work to oscillations around the positive steady states described in Lemma \ref{lemmeinstable}.
We show in Section~\ref{lll} that these solutions exhibit oscillations, by applying the Hopf bifurcation theorem. This behavior occurs only if the non-trivial steady state is unstable.
\end{remark}

For the sake of conciseness we define the following functions:
\begin{equation}
\left\{
	\begin{aligned}
T(k)&=\f{1}{\Lbar} \big( 2 k + \f{k + d_E}{b_E} (k+d_E - d_L) \big),\\
D(k)&=\f{1}{\Lbar} \f{k + d_E}{b_E d_E} \big( k (b_E - d_L) - d_E d_L \big).
\end{aligned}
	\right.
	\label{bkgk}
\end{equation}
We can rephrase Lemma \ref{lemmeinstable} into: Let $(k, k') = (h(\Lbar), h'(\Lbar))$ at some equilibrium $\Lbar$.
The state $(\Ebar, \Lbar)$ is unstable if and only if either $k' > T(k)$ or $T(k) \geq k' > D(k)$.
Thanks to \eqref{hyp:params} we can define
\begin{eqnarray}\label{eq:k+}
k_{\pm} :=\frac{d_E(b_E+2d_E+d_L) \pm \sqrt {4d_E^3(b_E-d_E-d_L)+d_E^2(b_E+2d_E+d_L)^2}}{2(b_E-d_E-d_L)}.
\end{eqnarray}

\begin{lemma}\label{k+}
Assume \eqref{eq:sufcond} holds.
If $k > k_+$ then $T(k) < D(k)$, and if $k\in(0,k_+)$ then $T(k) > D(k)$.
\end{lemma}

\begin{proof}
We are looking for the $k > 0$ such that $T(k) > D(k)$, that is also written from \eqref{bkgk}
\begin{eqnarray*}
k^2(b_E-d_E-d_L) - k d_E(b_E+2d_E+d_L)-d_E^3 < 0.
\end{eqnarray*}
Recalling that $b_E > d_E + d_L$ by \eqref{eq:sufcond}, the discriminant is:
\begin{eqnarray*}
\Delta=d_E^2(b_E+2d_E+d_L)^2+4d_E^3(b_E-d_E-d_L) > 0.
\end{eqnarray*}
The roots are exactly $k_{\pm}$, so the polynomial is negative when $k$ $\in $ $(k_-, k_+)$.

We note that $k_-<0$, so $T(k)<D(k)$ if and only if $k > k_+$, and $T(k) > D(k)$ if and only if $ k \in (k_-, k_+).$
Since $k>0$, this is equivalent to $k \in (0, k_+)$.
\end{proof}

Collecting our results on the equilibria we can state
\begin{proposition}
	Assume \eqref{eq:sufcond} holds, and let $(\Ebar, \Lbar)$ be a positive steady state of \eqref{eq:S}. Then $k_+ > \df{d_E d_L}{b_E - d_L}$ and necessarily $h(\Lbar) > \df{d_E d_L}{b_E - d_L}$.
	
	If $h(\Lbar) > k_+$, then $(\Ebar, \Lbar)$ is unstable if and only if $h'(\Lbar) > T\big( h(\Lbar) \big)$.
	If $\f{d_E d_L}{b_E - d_L} < h(\Lbar) < k_+$, then it is unstable if and only if $h'(\Lbar) > D \big( h(\Lbar) \big)$.
	
	Finally, the eigenvalues of the linearized of \eqref{eq:S} at $(\Ebar, \Lbar)$ are complex conjugate and pure imaginary if and only if $h(\Lbar) > k_+$ and $h'(\Lbar) = T \big( h(\Lbar) \big)$.
	\label{prop:eigenvalues}   
\end{proposition}
\begin{proof}
	This is a direct consequence of the previous calculations, except for
	\beq
	k_+ = \frac{d_E(b_E+2d_E+d_L) + \sqrt {4d_E^3(b_E-d_E-d_L)+d_E^2(b_E+2d_E+d_L)^2}}{2(b_E-d_E-d_L)} > \f{d_E d_L}{b_E - d_L}.
	\label{hyp:k+}
	\eeq
	Inequality \eqref{hyp:k+} is equivalent to
	\[
		\f{b_E - d_L}{b_E - d_L - d_E} \big(b_E + d_L + 2 d_E  + \sqrt{ (b_E + 2 d_E + d_L)^2 + 4 d_E (b_E - d_E - d_L)}\big) > 2 d_L.
	\]
This inequality holds because $b_E > d_L$ (thanks to \eqref{hyp:naturelle}).
Indeed,
\begin{align*}
&\f{b_E - d_L}{b_E - d_L - d_E} \big(b_E + d_L + 2 d_E  + \sqrt{ (b_E + 2 d_E + d_L)^2 + 4 d_E (b_E - d_E - d_L)}\big) 
\\
&> \big(b_E + d_L + 2 d_E  + \sqrt{ (b_E + 2 d_E + d_L)^2 + 4 d_E (b_E - d_E - d_L)}\big) \\
&> b_E + d_L \\
&> 2 d_L.
\end{align*}

	Then, setting $k = h(\Lbar)$, $k' = h'(\Lbar)$ and using the notations \eqref{eq:trdet}, the eigenvalues of the linearized operator are roots of the polynomial
	\[
		P(\lambda) = \lambda^2 - \lambda \tr(A) + \det(A).
	\]
	Hence the roots are pure imaginary if and only if $\tr(A) = 0$ and $\det(A) > 0$.
	From the definition of $T,D$ in \eqref{bkgk}, $\tr(A) = 0$ if and only if $k' = T(k)$. As $\det(A) > 0$ if and only if $k' < D(k)$, by Lemma~\ref{k+} this holds whenever $k > k_+$.
	
\end{proof}

\subsection{Discussion on the nonlinearities and the equilibrium values}
\label{discuss:nonlinearities}

We discuss in this paragraph the nonlinearities of system \eqref{eq:S}, and the role they play.

First we justify the use of a competition term. Solutions of \eqref{eq:S} are bounded (Lemma \ref{lem:SS}), but this holds only thanks to the nonlinear competition term $-c L^2$ in the equation describing the larvae dynamics.
More generally, any competition term $\phi(L)$, as in Section \ref{sec:reduction} such that $\phi(L) \to +\infty$ as $L \to +\infty$ yields the same result.
However, in the absence of such a competition, {\it a priori} bound on the solutions cannot be obtained, and no phenomenon keeps the population finite. For {\it Aedes} mosquitoes, the amount of available food in the breeding sites is an actual resource limitation that can trigger massive death of larvae if the amount of food per larva drops down too low (see \cite{Azn.Model}).
Therefore, we choose the simplest ({\it i.e.} quadratic) competition term to represent this competition for resources, and this ensures mathematically that solutions remain bounded.

Still, the competition parameter $c$ is extremely hard to assess from experimental data, and the values we use in this work should be handled with care. Usually, we fix a value for a positive equilibrium $\Lbar$ (which corresponds to choosing a type of breeding site).
Then, to each value $k = h(\Lbar)$ corresponds a non-necessarily unique $c(k)$ that makes $\Lbar$ an equilibrium of \eqref{eq:S}.
We treat $k$ as a free parameter in this study. 
It has been observed that the hatching rate indeed is extremely dispersed (see for instance the experimental results of \cite{Liv.Complex}), depending not only on the mosquito population and the environmental conditions but also on the egg batches themselves.
In future works expanding on the simplest oscillatory behavior we describe here, this variability in the actual value of $k$ should be taken into account if the model outputs are to be linked with experimental data.

Second, we discuss the hatching rate function $h$, which is crucial to our study. 
From now on, we require $h$ to be increasing. 
Indeed, Proposition \ref{prop:eigenvalues} shows that a steady state is always stable if $h$ is decreasing. Hence only an increasing $h$ can produce stable oscillations.
This mathematical assumption is supported by a simple biological hypothesis: larvae promote hatching.

An interesting feature of this intuition is that it can be subsequently extended to higher-dimensional systems such as \eqref{eq:SG}. In other words, it is not an artifact produced by considering only a 2-dimensional system but a robust qualitative property for these systems.

Indeed, for \eqref{eq:SG} the Jacobian matrix at any point $X = (E, L, P, A)$ reads 
\[
	J(X) = \begin{pmatrix}
			-\delta_E -h(L) & -E h'(L) & 0 & \beta_E \\
			h(L) & h'(L) E - \delta_L - \tau_L - 2 c L & 0 & 0 \\
			0 & \tau_L & -\delta_P - \tau_P & 0 \\
			0 & 0 & \tau_P & -\delta_A
		\end{pmatrix},
\]
hence if $h'(L) < 0$ then $J(X)$ is a Metzler matrix (it has positive extra-diagonal coefficients): the system is cooperative in this case.
Its characteristic polynomial may be written
\[
	P(\lambda) = (\lambda + A_1) (\lambda + A_2) (\lambda^2 + A_3 \lambda + A_4) - C,
\]
where $A_i, C > 0$. Being a Metzler matrix, $J$ has a real dominant eigenvalue. This matrix is stable if and only if this eigenvalue is negative; in other words, if and only if $P(0) > 0$ (since $P$ is increasing on $(0, +\infty)$). This condition reads
\[
	\delta_A (\delta_P + \tau_P) \big( (\delta_E + h(L)) (-h'(L) E + \tau_L +\delta_L + 2 c L) + E h(L) h'(L) \big) > \beta_E \tau_L \tau_P h(L).
\]
At equilibrium,
\[
\delta_L + \tau_L + c L = \f{h(L)}{h(L) + \delta_E} \f{\beta_E \tau_L \tau_P}{\delta_A (\delta_P + \tau_P)},
\]
therefore $P(0) > 0$ and thus any equilibrium where $h'<0$ must be (locally) stable, in system \eqref{eq:SG} as well as in system \eqref{eq:S}.
Adding ``neutral'' compartments keeps this property true and we can be confident in concluding that only a positive effect of larvae on hatching rate can destabilize the equilibrium and lead to (local) oscillations.

Some preliminary experiments ran by one of the authors seem to indicate that the larval impact on hatching may depend on larval development stage. Taking this into account would require model complexification. For instance, to model hatching impact discrepancies between first instar (positive) and last instar larvae (negative) we could add at least one compartment in \eqref{eq:S}.
However, we focus here on the simplest oscillations-producing mechanism.
The hatching function being increasing and bounded, it is reasonable to assume that $h$ is S-shaped and smooth, which is what we use in the rest of the paper.

Third, having discussed the two nonlinearities in \eqref{eq:S}, we are left with an important question about steady states: how to ensure that $\Lbar$ is actually unique? The second equation in \eqref{eq:steady} is also written
\beq
	h(\Lbar) = d_E \df{d_L + c \Lbar}{b_E - d_L - c \Lbar}.
	\label{eq:hss}
\eeq
The number of positive steady states depends strongly on function $h$. Being a S-shaped function does not guarantee uniqueness.
Therefore, it should be checked case by case except for some simple function families. We illustrate this fact in Appendix \ref{app:Hill} with Hill functions.
Still, we notice that $\kappa : L \mapsto d_E \df{d_L + c L}{b_E -d_L - c L}$ is convex on $(0, (b_E - d_L)/c)$ and goes to $+ \infty$ at $(b_E - d_L)/c$. So for instance uniqueness is guaranteed if \eqref{hyp:naturelle} holds and either, for all $L \in (0, (b_E - d_L)/c)$, $h'' (L) < 0$ or
\[
 h'(L) < \kappa'(L) = \df{d_E c b_E}{(b_E - d_L - c L)^2}.
\]

\section{The slow-fast oscillatory regime}
\label{sec:slow-fast}

In order to understand periodic solutions to \eqref{eq:S}, we examine a possible regime with a small parameter and then prove the oscillation result (Theorem \ref{thm:mainslowfast}). We have in mind here the analysis of the FitzHugh-Nagumo system. Numerical illustration, amplitude and period computation in some particular cases can be found in Appendix \ref{app:slowfast}.

\subsection{Parameter regime and main result}

Here, we assume that the egg stock is large, and its dynamics slow compared with the larvae stock. This identifies a small parameter leading to a slow-fast system.

More precisely, let $\e > 0$, $\eta : \R_+ \to \R_+$, and assume at first that all parameters (except for $h$) may depend on $\e$. We transform the variables $(E, L)$ from \eqref{eq:S} into $v_{\e} := \e E$ and $u_{\e} := \f{1}{\eta (\e)} L$. These new variables satisfy
\beq
\bepa
 \dot{v}_{\e} = \e \eta(\e) b_E u_{\e} - \big( d_E + h(\eta (\e) u_{\e}) \big) v_{\e} =: f_{\e} (u_{\e}, v_{\e}), 
 \\[10pt]
 \e\dot{u}_{\e} = \f{1}{\eta (\e)} h(\eta(\e) u_{\e}) v_{\e} - d_L \e u_{\e} - c \eta(\e) \e u^2_{\e} =: g_{\e} (u_{\e}, v_{\e}).
\eepa
\label{sys:uveps}
\eeq
We assume that parameters scale in such a way that the following limits exist, as $\e \to 0$:
\beq
\bepa
\feps \xrightarrow{L^{\infty}} f, \quad \geps \xrightarrow{L^{\infty}} g,
\\[10pt]
u_{\e} (t=0) = u_{\e}^0 \xrightarrow{} u_0, \quad v_{\e} (t=0) = v_{\e}^0 \xrightarrow{} v_0.
\eepa
\label{eq:fglimites}
\eeq
In addition, we assume that the zero set of $g$ is ``non-degenerate'' in the sense:
\beq
\forall v \geq 0, \quad \big\{ \sigma \geq 0, \, g(\sigma, v) = 0 \big\} \text{ does not contain any open interval.}
\label{ass:gnotvanish}
\eeq

We give below a simple proof of the following fact, in the spirit of Tikhonov's theorem on dynamical systems \cite{JP}.
\begin{theorem}
Consider system \eqref{sys:uveps} with $d_E, d_L, \Lbar$ and $h$ fixed, $b_E(\e) = \f{h(\Lbar) + d_E}{\e \Lbar}$, $\eta (\e) = \f{\Lbar^2}{h(\Lbar) - \e d_L \Lbar}$, for $\e$ small enough, and $c_{\e} = \f{1}{\e \eta (\e)}$. Let $\Ebar(\e) := 1 / \e$. Then $(\e \Ebar(\e), \f{1}{\eta(\e)}\Lbar) = (1, \f{h(\Lbar) - \e d_L \Lbar}{\Lbar})$ is a steady state of \eqref{sys:uveps} for all $\e > 0$ and \eqref{eq:fglimites} holds.

In addition, solutions of system \eqref{sys:uveps} along with any bounded initial data admits a limit as $\e \to 0$: there exists $u, v \in L^1 \cap L^{\infty} (0, T)$ for all $T > 0$ such that $v_{\e} \to v$ uniformly and $u_{\e} \to u$ in $L^p (0, T)$ for all $p< \infty$.

Moreover, if initial data $u^0_{\e}, v^0_{\e}$ are such that $\big( sgn(\geps), sgn(\feps) \big) (u^0_{\e},v^0_{\e})$ is constant for $\e$ small enough, then $(u, v)$ is periodic, $g(u(t), v(t)) = 0$ for almost every $t > 0$ and the trajectory is uniquely defined from $f$ and $g$ with $\f{dv}{dt} = f(u, v)$.
\label{thm:mainslowfast}
\end{theorem}

\begin{figure}[h!]
 \includegraphics[width=.5\linewidth]{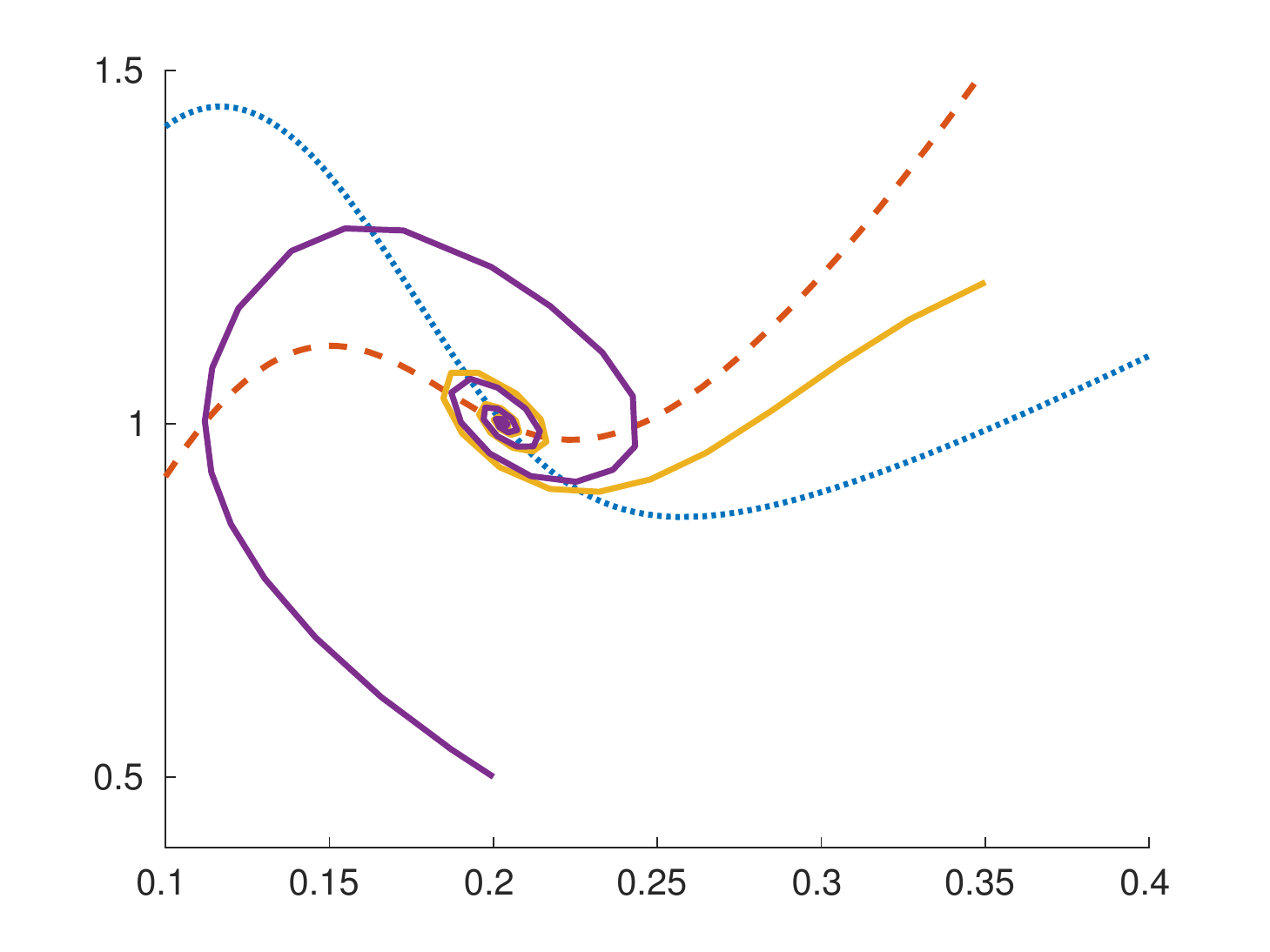}
 \includegraphics[width=.5\linewidth]{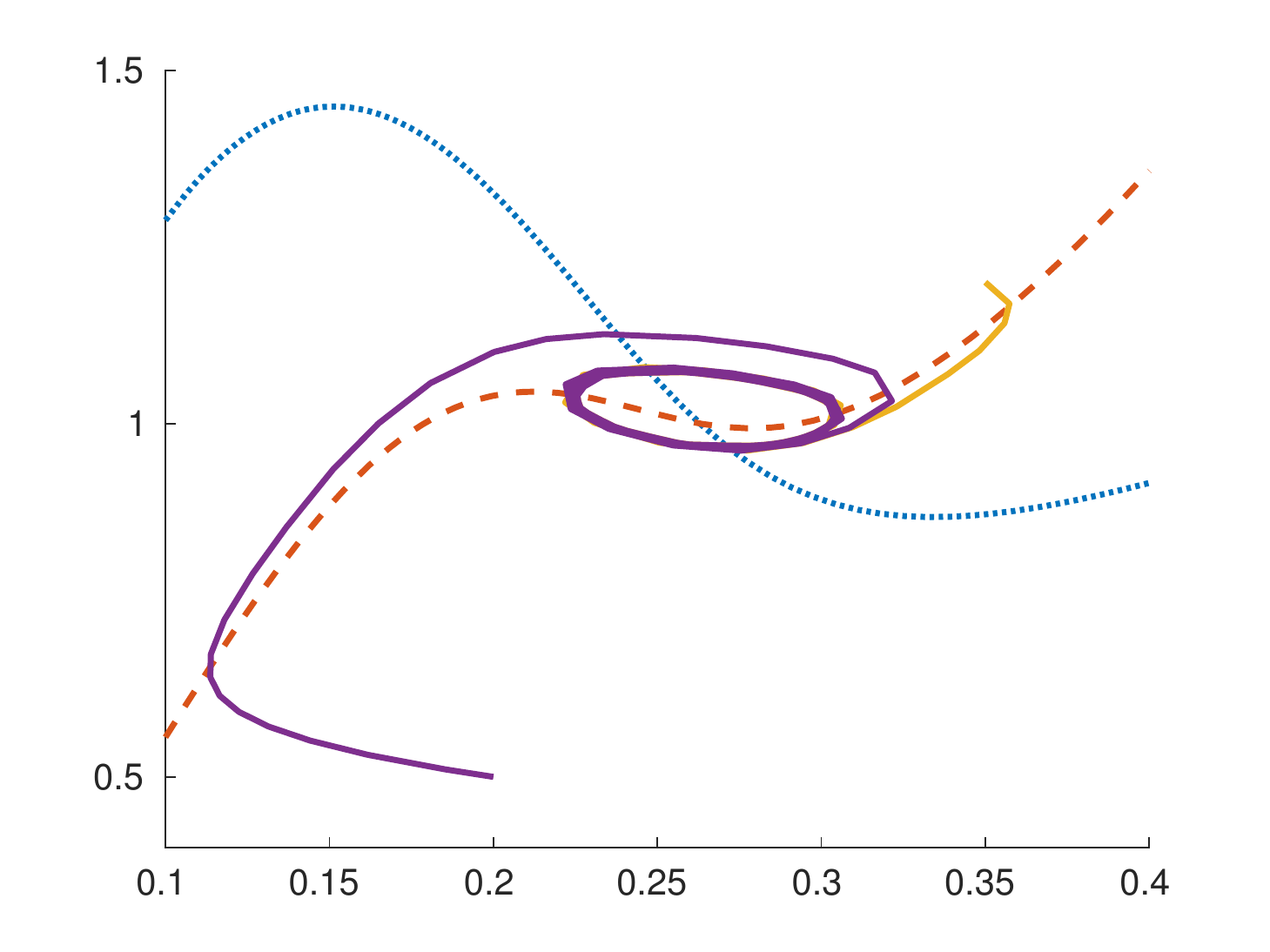}
 \includegraphics[width=.5\linewidth]{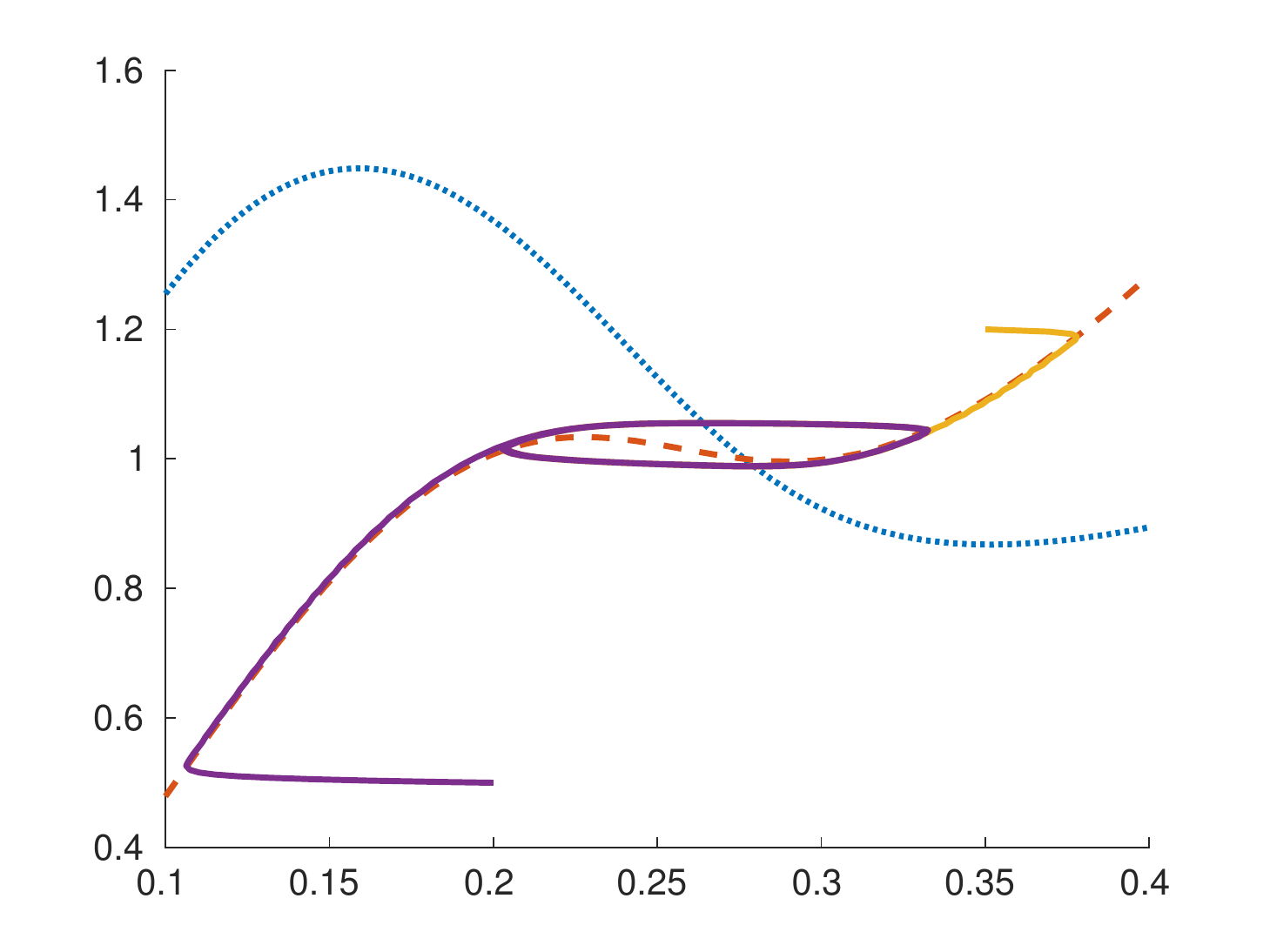}
 \includegraphics[width=.5\linewidth]{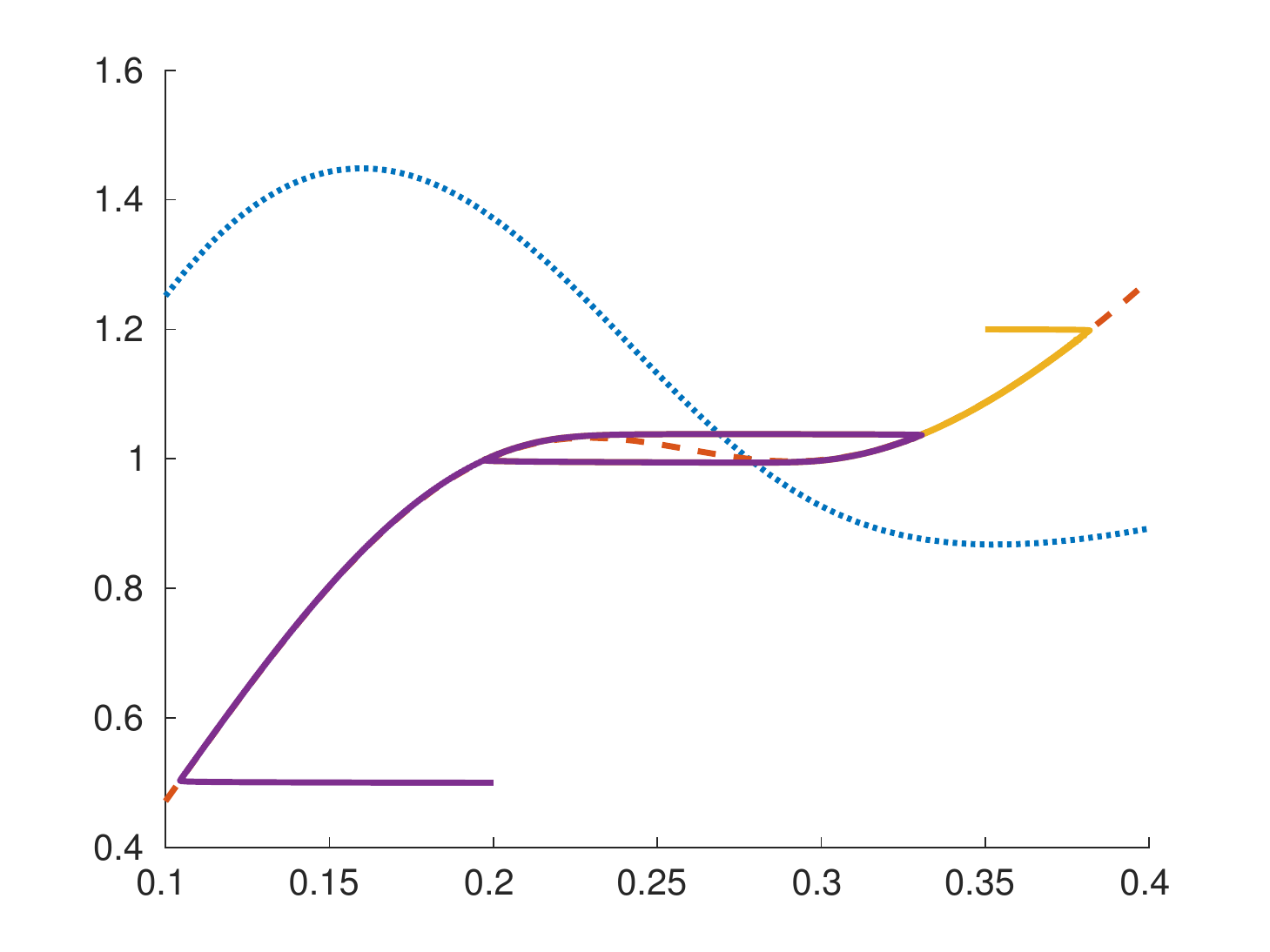}
 \caption{$u$ is in $x$-axis, $v$ in $y$-axis. Red dashed curves correspond to nullclines $\geps= 0$ ($\dot{\ueps} = 0$)
 and blue dotted curves to nullclines $\feps=0$ ($\dot{\veps} = 0$).
 The four figures correspond to decreasing values of $\e$ from top-left to bottom-right ($0.5$, $0.1$, $0.01$ and $0.001$). In yellow and purple, two trajectories $t \mapsto (\ueps(t), \veps(t))$ are shown, for two different initial conditions (respectively $(0.35, 1.2)$ and $(0.2, 0.5)$).}
 \label{fig:5}
\end{figure}

Figure \ref{fig:5} illustrates the slow-fast dynamics.
Before proving Theorem \ref{thm:mainslowfast}, we justify the particular scaling choices in its statement.
Non-trivial equilibrium $(\Ebar, \Lbar)$ of \eqref{eq:S} are given by \eqref{eq:steady}:
\[
d_L + c \Lbar = \f{h(\Lbar) b_E}{h(\Lbar) + d_E}, \quad \Ebar = \f{b_E \Lbar}{h(\Lbar) + d_E}.
\]
Thus in all generality (allowing all parameters to depend on $\e$), the scalings fit for our purpose ({\it i.e.} with $\Ebar (\e) = 1 /\e$) are exactly those for which $\e = \f{h (\Lbar (\e)) + d_E (\e)}{b_E (\e) \Lbar(\e)}$ and there exists $\eta (\e) = O(1)$ such that
\[
  d_L(\e) \big( h (\Lbar(\e)) + d_E (\e) \big) + \f{b_E (\e) \Lbar^2 (\e)}{\eta (\e)} = h (\Lbar(\e)) b_E (\e).
\]
It turns out that $\Lbar(\e) \big( \e(b_E - d_L) - \f{1}{\eta} \big) = d_E$.
Hence to guarantee $\eta(\e) = O(1)$ it is required that
\[
	b_E - d_L = O(1/\e).
\]
Therefore the scaling choice made in Theorem \ref{thm:mainslowfast} is in some sense ``generic''.

Note that for every possible parameter scaling we get a (possibly different) limit in \eqref{eq:fglimites}.
For instance, assuming $\e c_{\e}$ and $\e b_E (\e)$ have limits $1/\eta_0, \xi > 0$ respectively as $\e \to 0$ (this is the case with the scaling used in Theorem \ref{thm:mainslowfast}), we  choose $\eta (\e) = O(1)$ such that $c_{\e} \eta (\e) \e = 1$ and end up with
\[
 \bepa
 \dot{v} = \e \eta b_E u - \big(d_E + h(\eta u) \big) v =: \feps(u, v),
 \\[10pt]
 \e \dot{u} = \f{1}{\eta} h (\eta u ) v - d_L \e u - u^2 =: \geps(u, v).
 \eepa
\]
The limits $f$ and $g$ are given by
\beq
 f(u, v) = \eta_0 \xi u - \big( d_E + h(\eta_0 u) \big) v,
\quad
 g(u,v) = \f{1}{\eta_0} h(\eta_0 u) v - u^2.
\label{eq:fg}
\eeq

\subsection{Proof of the main result}
We proceed to the proof of Theorem \ref{thm:mainslowfast} in three steps.
First, scaled quantities $u_{\e}$ and $v_{\e}$ remain uniformly bounded independently of $\e$, as can be proved from direct computation using the bound $K$ from Lemma \ref{lem:SS}.
\begin{lemma}
 There exists $C > 0$ such that for all $\e > 0$ and $t > 0$,
 \[
  \lvert \ueps(t) \rvert, \, \lvert \veps(t) \rvert, \, \lvert \feps(\ueps(t), \veps(t)) \rvert, \, \lvert \geps(\ueps(t), \veps(t)) \rvert \, \leq C.
 \]
 \label{lim:trivbound}
\end{lemma}
Hence, up to extraction, $\veps$ converges to $v$ uniformly on compact sets $[0, T]$ by the Ascoli theorem.
Then, the convergence of an auxiliary quantity gives convergence of $\ueps$:
\begin{lemma}
For all $T > 0$.
 \beq
  \lVert \geps( \ueps, \veps) \rVert_{L^2 (0, T)} = O(\sqrt{\e}).
  \label{estim:sqrte}
 \eeq

 Moreover, there exists $u, v \in L^1 \cap L^{\infty}$ such that after extraction of a subsequence $\ueps \to u$ in $L^p(0, T)$ for all $1 \leq p<\infty$,
 as $\veps \to v$ uniformly.
 \label{lem:auxiliary}
\end{lemma}
\begin{proof}
Let $B(t, u) := \int_0^u g^2 (\sigma, v(t)) d\sigma$,
where $v$ is the limit of $\veps$ (obtained by the Ascoli theorem) and $g$ is the limit of $\geps$ (from \eqref{eq:fglimites}).
From~\eqref{ass:gnotvanish} we deduce that for all $t$, $u \mapsto B(t, u)$ is increasing. Hence there exists a {\it smooth} function $A(t, u)$ such that for all $t, u$, $A(t, B(t, u) ) = u.$

If there exists $w(t) \in L^p (0, T)$ for all $p < \infty$ and $T > 0$ such that
\beq
 \int_0^{\ueps(t)} \geps^2 (\sigma, \veps(t) ) d\sigma \xrightarrow[\e \to 0]{L^p (0, T)} w(t),
 \label{exists:limit}
\eeq
then defining $u(t) := A(t, w(t))$ we can conclude that $ \ueps = A\big( \cdot, \int_0^{\ueps} g^2 (\sigma, v) d\sigma \big)  \xrightarrow[\e \to 0]{L^p (0, T)} u = A(\cdot, w).$

Indeed, we notice that
\[
 \int_0^{\ueps(t)} \geps^2 (\sigma, \veps (t) ) d \sigma - \int_0^{\ueps(t)} g^2 (\sigma, \veps (t) ) d \sigma \to 0,
\]
and
\[
 \int_0^{\ueps(t)} g^2 (\sigma, \veps (t) ) d \sigma - \int_0^{\ueps(t)} g^2 (\sigma, v (t) ) d \sigma\to 0.
\]
Since $\ueps$ is uniformly bounded,
\[
 \big \lvert \int_0^{\ueps(t)} \geps^2 (\sigma, \veps (t) ) d \sigma - \int_0^{\ueps(t)} g^2 (\sigma, v (t) ) d \sigma \big\rvert \leq \ueps(t) \big( \lVert \geps^2 - g^2 \rVert_{\infty} + C \lVert \veps - v \rVert_{\infty} \big),
\]
for some $C > 0$ which depends only on $\p_v g$.
Hence \eqref{exists:limit} implies
\[
 \int_0^{\ueps(t)} g^2 (\sigma, v(t)) d\sigma \xrightarrow[\e \to 0]{L^p (0, T)} w(t).
\]

Therefore we only need to prove \eqref{exists:limit} to complete the proof. To do so we first obtain \eqref{estim:sqrte} by computing
\begin{align*}
 \f{ \geps(\ueps(t), \veps(t))^2}{\e} &= \geps( \ueps(t), \veps(t)) \dot{\ueps} \\
 &= \f{d}{dt} \int_0^{\ueps (t)} \geps( \sigma, \veps(t) ) d\sigma - \feps(\ueps, \veps) \int_0^{\ueps(t)} \p_v \geps (\sigma, \veps(t))d\sigma.
\end{align*}

Hence
\[
 \f{1}{\e}\int_0^T \big( \geps( \ueps(t), \veps(t) ) \big)^2 dt = \int_{\ueps(0)}^{\ueps(T)} \geps(\sigma, \veps(t)) d\sigma - \int_0^T \feps(\ueps(t), \veps(t) ) \int_0^{\ueps(t)} \p_v \geps(\sigma, \veps(t))d\sigma dt.
\]

Since $\feps, \geps$ and $\p_v \geps = \f{1}{\eta (\e)} h(\eta(\e) u_{\e})$ are uniformly bounded, we deduce that
\[
 \int_0^T \geps( \ueps(t), \veps(t) )^2 dt = O (\e).
\]
This gives \eqref{estim:sqrte}. Then we introduce
\[
 \weps(t) := \int_0^{\ueps(t)} \geps^2 (\sigma, \veps(t)) d\sigma.
\]
We compute
\[
 \dot{\weps} (t) = \f{1}{\e} \geps^2 (\ueps(t), \veps(t)) \e \dot{\ueps} + \feps(\ueps(t), \veps(t)) \int_0^{\ueps(t)} 2 \geps (\sigma, \veps(t)) \p_v \geps (\sigma, \veps(t) )d\sigma.
\]

By the previous point, $t \mapsto \f{1}{\e} \geps^2 (\ueps(t), \veps(t))$ is uniformly (in $\e$) bounded in $L^1$.
In addition, $t \mapsto \e \dot{\ueps}(t)$ is uniformly (in $\e$) bounded in $L^{\infty}$, by the Lemma \ref{lim:trivbound}.
The second term $\feps \int \geps \p_v \geps$ is uniformly bounded as well.

As a consequence, $\weps$ is uniformly (in $\e$) bounded in $BV_{\text{loc}}$.
This implies that up to extraction, $\weps \to w$ in $L^1$. Because $\weps$ is also bounded in $L^{\infty}$,
convergence actually takes place in all $L^p$ spaces.

\end{proof}

Finally, the shapes of $(f, g)$ allow us to describe simply the limit trajectories. We use the following assumptions:
for all $\e > 0$ small enough, we assume that the right-hand sides of system \eqref{sys:uveps} satisfy
  \begin{enumerate}
   \item[(R.1)] the set $\R^2 \backslash \{ f_{\e}=0, g_{\e} = 0 \}$ has exactly $4$ connected components, whose measures do not vanish as $\e \to 0$,
   \item[(R.2)] $f(u_0, v_0) \not=0$, $g(u_0, v_0) \not= 0$ and the couple $\big( \sgn(f_{\e} (u^{\e}_0, v^{\e}_0), \sgn(g_{\e} (u^{\e}_0, v^{\e}_0) \big)$ is constant and equal to $\big(\sgn(f (u_0, v_0)), \sgn(g(u_0, v_0)) \big)$.
  \end{enumerate}
  We also assume that the uniform limits $f,g$ of $f_{\e}, g_{\e}$ satisfy
  \begin{enumerate}
   \item[(L.1)] the curve $\Upsilon := \{ g = 0 \}$ is the graph of a function $\phi \in \mathcal{C}^1 (\R_+, \R_+)$ with $\phi(\infty) = \infty$ and $\phi(0) = 0$,
   \item[(L.2)] the function $g$ is positive on the epigraph of $\phi$,
   \item[(L.3)] the function $\phi$ has exactly two local extrema,
   \item[(L.4)] on the graph of $\phi$, $\sgn(f) = -1$ except for a bounded set.
  \end{enumerate}

\begin{lemma}
 With these assumptions we have:
  
 There exists a unique $\tau > 0$ and a (unique up to translations) $\tau$-periodic function $(u_{\tau}, v_{\tau}) : \R_+ \to \Upsilon$ such that $v_{\tau}$ is Lipschitz-continuous, $u_{\tau}$ is piecewise continuous, for all $t \geq 0$, $v_{\tau} = \phi(u_{\tau})$ everywhere, $\dot{v}_{\tau} = f(u_{\tau}, v_{\tau})$ almost everywhere and the discontinuities of $u_{\tau}$ are located at times $t$ such that $\phi$ has a local extremum at $u_{\tau} (t^-)$.

 There exists $\tau_1 \geq 0$ and $\tau_2  \in [0, \tau)$ such that for all $t > \tau_1$, $(u,v) (t) = (u_{\tau}, v_{\tau}) (t + \tau_2)$.
  Moreover, by construction $\tau_1$ and $\tau_2$ are uniquely defined from $u_0$ and $v_0$, so the limit $(u,v)$ is in fact unique and the whole family $(u_{\e}, v_{\e})_{\e}$ converges as $\e$ goes to $0$.

  \label{lem:trajectories}
\end{lemma}

Clearly from \eqref{eq:fg}, Lemma \ref{lem:trajectories} applies with the hypotheses of Theorem \ref{thm:mainslowfast} and
\beq
 \phi(u) = \f{\eta_0 u^2}{h(\eta_0 u)}, \quad \eta_0 = \f{\Lbar^2}{h(\Lbar)},
 \label{eq:phieta}
\eeq
thus proving the remaining part of the theorem. 

\begin{proof}[Proof of Lemma \ref{lem:trajectories}]

Thanks to assumptions (R.1), (L.1), (L.3) and (L.4), the construction of $(u_{\tau}, v_{\tau})$ is classical and can be done by pasting together solutions of Cauchy problems given (locally) by $\dot{v}_{\tau} = f( \phi^{-1} (v_{\tau}), v_{\tau})$, on intervals where $\phi$ is invertible. Uniqueness comes from the crucial fact that discontinuities of $u_{\tau}$ are assumed to be located at local extrema of $\phi$.

From the previous lemmas we know that $(u, v) \in \Upsilon$ almost everywhere.
In addition, uniform boundedness of $f_{\e}(u_{\e}, v_{\e})$ ensures that $v$ is Lipschitz continuous.

Then, we claim that if $t > 0$ is such that $\phi$ has no local extremum at $u(t^+)$, then there exists $\tau_0 > 0$ such that $(u, v)$ is continuous on $(t, t + \tau_0)$. This point is the key of the proof.
To prove it, let $u_i$ be such that $\phi' (u_i) < 0$. We solve only the simpler problem 
\[
\bepa
\dot{\widehat{v}}_{\e} = f (\widehat{u}_{\e}, \widehat{v}_{\e}), \quad \widehat{v}_{\e} (0) = \phi(u_i) + O(\e),
\\[10pt]
\e \dot{\widehat{u}}_{\e} = g(\widehat{u}_{\e}, \widehat{v}_{\e}), \quad \widehat{u}_{\e} (0) = u_i + O(\e).
\eepa
\]
Introducing $\widehat{w}_{\e} := \widehat{u}_{\e} - \phi^{-1} (\widehat{v}_{\e})$, where the inverse of $\phi$ is taken locally (this is possible for $\e$ small enough since $\phi'(u_i) < 0$ and $\widehat{v}_{\e}$ is uniformly Lipschitz-continuous), we obtain
\[
 \dot{\widehat{w}}_{\e} = \f{\widehat{w}_{\e}}{\e} \p_1 g (\widehat{r}_{\e} (t), \widehat{v}_{\e} (t)) - \f{f(\widehat{u}_{\e}, \widehat{v}_{\e})}{\phi' (\phi^{-1} (\widehat{v}_{\e}))}, \quad \widehat{w}_{\e} (0) = O(\e),
\]
for some $\widehat{r}_{\e} (t)$ between $\widehat{u}_{\e} (t)$ and $\phi^{-1} (\widehat{v}_{\e} (t))$.
We have $\p_1 g \leq - \alpha < 0$ on a neighborhood of $(u_i, \phi(u_i))$, so on this neighborhood $\widehat{w}_{\e}$ remains small (it is a $o(\e)$), which in turn proves that $(\widehat{r}_{\e}, \widehat{v}_{\e})$ remains in this neighborhood. In particular, $\widehat{u}_{\e}$ converges to some function $\widehat{u}$ which is continuous at $t = 0$ (since it is equal to $\phi^{-1} (\widehat{v}(t))$ on a positive neighborhood of $0$).
We do not write the full proof because the derivation we use here extends readily at the price of tedious notations. A full proof should use $f_{\e}, g_{\e}$ rather than $f, g$, and rise some analogue $\phi_{\e}$ of $\phi$ at level $\e > 0$, for $\e$ small enough, which is locally invertible on a neighborhood of the initial data. It does not require more assumptions than the ones we stated.

This is enough to get all the results of Lemma \ref{lem:trajectories}, except for the initial layer which we treat now.
To fix the notations, we assume that $\phi$ has a local minimum equal to $\phi_m$ at $u_m$ and a local maximum equal to $\phi_M > \phi_m$ at $u_M < u_m$. Moreover, let $u_m^0 < u_M$ such that $\phi(u_m^0) = \phi(u_m)$. For $\alpha, \beta \in \{ 1, -1 \}$, we also introduce $Z_{\alpha}^{\beta} := \{ \sgn(f) = \alpha, \, \sgn(g) = \beta \}$.

We define a mapping $\pi : \R^2 \to \Upsilon$ by $\pi = Id$ on $\Upsilon$ and if $(u, v) \in Z_{\alpha}^{\beta}$ then $\pi (u, v) = (u_1, v)$ such that $\phi(u_1)=v$ and $\sgn(u_1 - u) = \beta$. The projection $\pi$ is well-defined thanks to the assumptions on $\phi$ and~$g$, except on $(u_m^0, +\infty) \times \{ \phi_m \}$, on which we let $\pi \equiv (u_m^0, \phi_m)$.
Then $(u, v) (0^+) = \pi(u_0, v_0)$. To prove this, one simply has to check the behavior of $\ueps$ (since $\veps$ and $v$ are Lipschitz continuous). As above, we claim that the first-order behavior is simply given by the ``layer equation''
\[
 \e \dot{\widetilde{u}}_{\e} = g(\widetilde{u}_{\e}, v_0), \quad \widetilde{u}_{\e} (0) = u_0,
\]
which makes $\widetilde{u}_{\e}$ converge exponentially fast to $\pi(u_0,v_0)_1$, thanks to assumptions (R.2) and (L.2).
Up to tedious notations and thanks to \eqref{eq:fglimites} and (R.2), this result extends to $u_0^{\e}$, $v_{\e}$ and $g_{\e}$.

Let $\Upsilon_u = \Upsilon \cap \big( [u_M, u_m] \times \R_+ \big)$ and $\Upsilon_s = \Upsilon - \Upsilon_u$. (Note that $\pi (\R_+^2 - \Upsilon_u) = \Upsilon_s$.) 
After the initial layer, the trajectory of $(u,v)$ remains on $\Upsilon_s$. This follows from the sign of $f$ on $\Upsilon_u$: because of the continuity property, the trajctory cannot exit $\Upsilon_s$ but at $(u_m, \phi_m)$ (or $(u_M, \phi_M)$, respectively). At these points however, $\Upsilon_u$ is repulsive since $v$ must be continuous, $\dot{v} < 0$ ($\dot{v} > 0$, respectively) and $\Upsilon_u$ lies locally in $\{ v > \phi_m \}$ (respectively in $\{ v < \phi_M \}$).

Still, the initial data does not need to be projected directly by~$\pi$ on $\Upsilon_s \cap \R_+ \times [\phi_m, \phi_M]$.
Therefore, we introduce $\tau_1 \geq 0$ as
\[
 \tau_1 := \max \big(0, \sup \{ t \geq 0, \quad v(t) \not\in [\phi_m, \phi_M] \} \big).
\]
It remains to check that $\tau_1 < +\infty$.
For all $T > 0$, as long as $\phi$ has no local extremum at $u(t)$ for $t \in (0, T)$, $u$ is continuous.
Thanks to our assumption (R.1), there are two connected components in $\Upsilon_s$, on each one of whom $\sgn(f)$ is constant. Because of assumption (L.4), $f$ must be negative on the unbounded connected component. Therefore $(u, v)$ remains on $(0, T)$ in a part of $\Upsilon$ where $\lvert f \rvert$ is positively bounded from below (one of the two connected components of $\Upsilon_s$) and has the appropriate sign. This yields the existence of $\tau_1 < +\infty$.

Then for all $t \geq \tau_1$ we have $v(t) \in [\phi_m, \phi_M]$, and the trajectory is uniquely defined onwards.
\end{proof}

\begin{remark}
 We did not treat the case when the limit of $(u_0^{\e}, v_0^{\e})$ belongs to $\Upsilon$ (relaxing assumption~(R.2)). In this case indeed, no general result can be obtained, unless the various convergence speeds (of~$\feps, \geps, u_0^{\e}$ and $v_0^{\e}$) are quantified.
\end{remark}

\begin{remark}
	The last point of Theorem \ref{thm:mainslowfast} implies that the amplitude of the oscillations (in $u, v$) at the limit $\e \to 0$ can be computed if one knows these parameter scale in $\e$ thanks to only $f$ and $g$.	
Their period $\tau$ can also be computed directly from $f$ and $\phi$. As in the proof of Lemma \ref{lem:trajectories} we denote the intervals of values taken by $u(t)$ where it is continuous (and thus $\mathcal{C}^{\infty}$) as $[u_m^0, u_M]$ and $[u_m, u_M^0]$ respectively, and let
\[
 \Psi (u, v) = \int_{u}^{v} \frac{\phi'(u')}{f(u',\phi(u'))}du'.
\]
Then we have
\beq
 \tau = \Psi(u_m^0,u_M) + \Psi (u_M^0, u_m).
 \label{formula:tau}
\eeq
\end{remark}

\section{Hopf bifurcation}

\label{lll}
Numerical observations (see Section \ref{app:Hill} and Appendix \ref{app:Hill})  show that the system (\ref{eq:S}) has a stable periodic solution oscillating around the non-zero steady state, even far from the slow-fast asymptotics we explored in the previous section.
We now prove the local existence of this periodic solution using the Hopf bifurcation theorem (Theorem 8.8 from \cite{forme}, with a classical proof in \cite{cracken}; see also \cite{JP}) for $2 \times 2$ systems of differential equations.

\subsection{The function class $H_{\Lbar}$}

To find out a possible bifurcation parameter, we choose the hatching function $h$ within a special class, for which we fix the value of one specific steady state $\Lbar$. 
With this setting, we can state a bifurcation theorem using the simple bifurcation parameter $h'(\Lbar)$, which represents the sensitivity of hatching rate to larval density at equilibrium.

However, it is worth noting that our argument does not rely on the structure of this class of functions, and may be adapted, for instance, to the Hill functions considered in Appendix \ref{app:Hill}.

For a fixed $\Lbar$ the class of functions under consideration that fits our purposes is
\begin{equation}
  H_{\Lbar} := 
  \Big\{ 
h(L)=a\Big(\arctan(b(L-\Lbar))+\frac{\pi}{2}\Big), \, a,b \in \R^+ 
    \Big\} .
    \label{eq:H}
\end{equation}
Graphs of these functions are shown in Figure \ref{fig2}.
\begin{figure}[h!]
 \includegraphics[width=.5\linewidth]{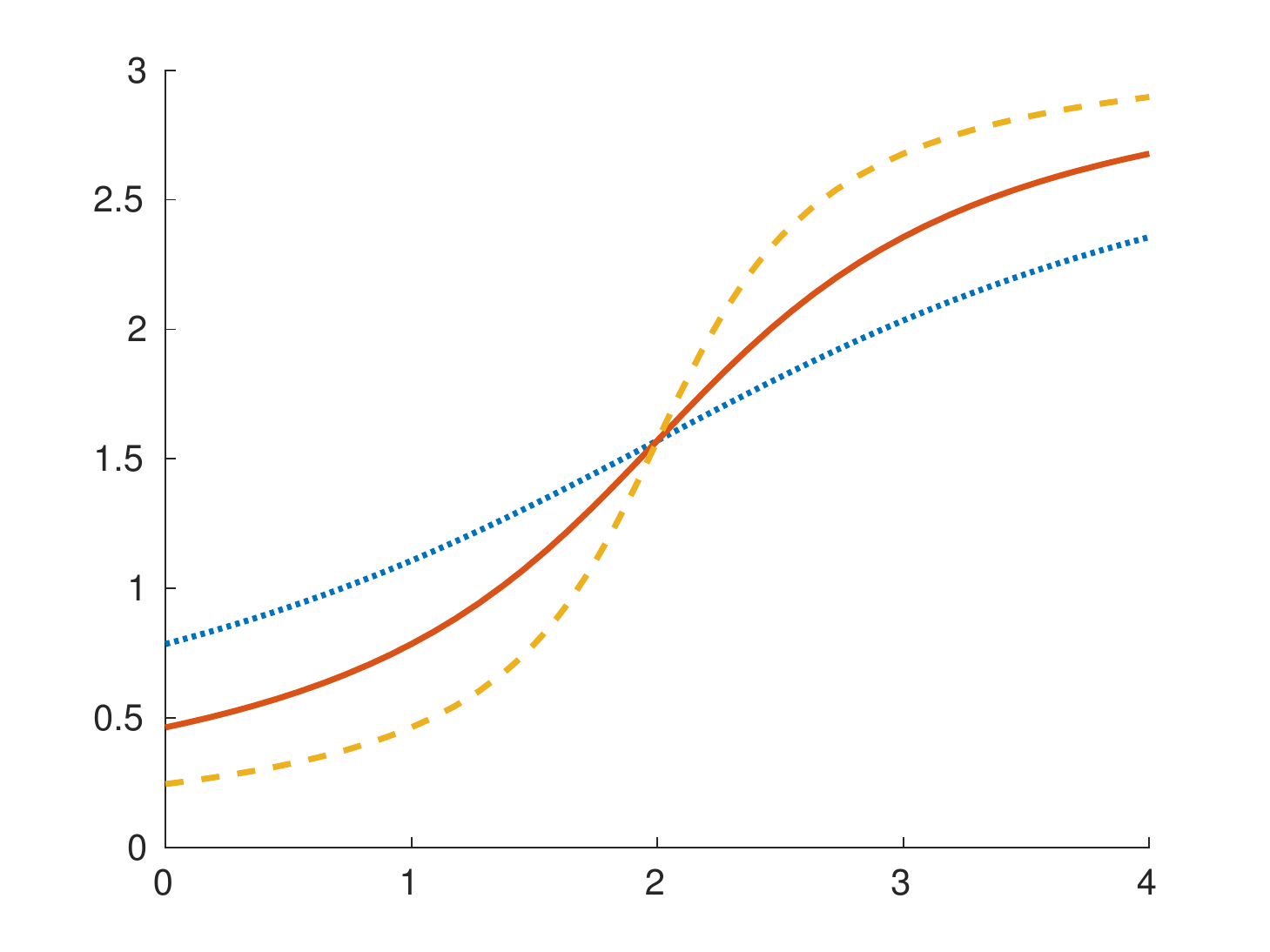}
 \includegraphics[width=.5\linewidth]{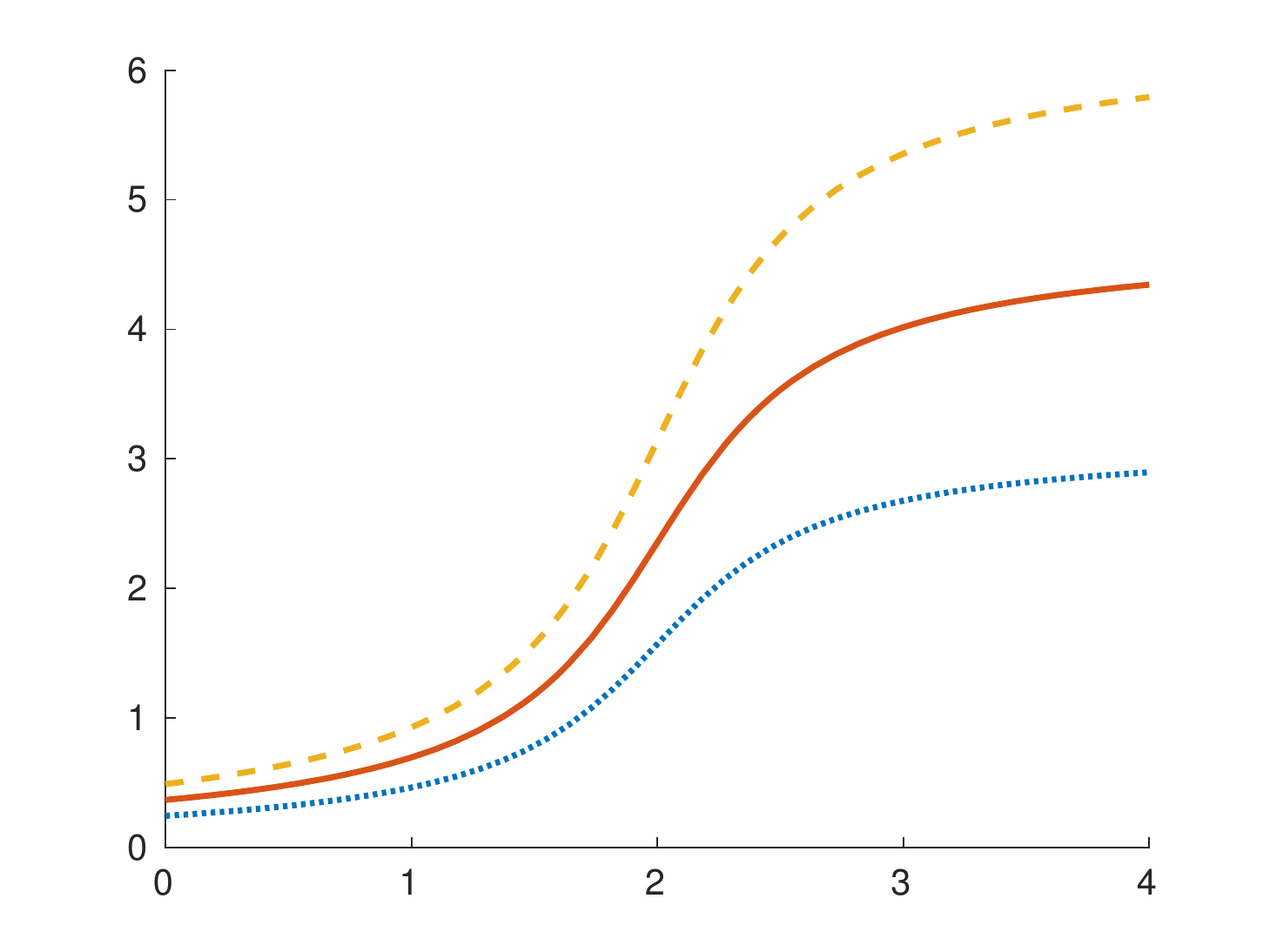}
 \caption{Function $h \in H_{\Lbar}$ with $\Lbar=2$. 
 {\it Left}: $a=1$, $b=\{0.5, 1, 2\}$. {\it Right}: $a=\{1, 1.5, 2\}$, $b=2$. Curve styles with increasing values in $a$ and $b$: dotted blue, solid red, dashed yellow.}
 \label{fig2}
\end{figure}
We use the immediate properties that these functions
are positive and increasing.
For any couple $(k,k')$ $\in \R_+^*\times\R_+^*$, there exists a unique function $ h$ of class $H_{\Lbar}$ with $h(\Lbar)=k$ and $h'(\Lbar)=k'$.
Finally, for all $c>0$, the steady state relation $h(\Lbar)=\frac{d_E(d_L+c\Lbar)}{b_E-d_L-c\Lbar}$ has a positive solution in $\Lbar$ if $a>\frac{d_Ed_L}{b_E-d_L}\frac{2}{\pi}$.
Indeed, for given values $(k, k') \in \R_+^2$, the choice of $a=\frac{2k}{\pi}$ and $b=\frac{k'}{a}$ gives the solution since
\[
h(\Lbar)= a\frac{\pi}{2}=\frac{2k}{\pi}\frac{\pi}{2}=k
\text{ and }
h'(\Lbar)=ab =\frac{2k}{\pi}\frac{k'}{\frac{2k}{\pi}}=k'.
\]
Also we can solve the equation in $\Lbar$, $\frac{a\pi}{2}=\frac{d_E(d_L+c\Lbar)}{b_E-d_L-c\Lbar}$, which yields $\Lbar = \frac{\frac{a\pi}{2}(b_E-d_L)-d_Ed_L}{c\frac{a\pi}{2}+cd_E}.$
Hence $\Lbar$ is positive under the stated condition.

\begin{remark}
From Lemma \ref{condh0}, for $h$ of class $H_{\Lbar}$, the state $(0,0)$ is unstable if and only if 
\begin{equation*}
a>\frac{d_Ed_L}{(b_E-d_L)(\frac{\pi}{2}+\arctan(-b\Lbar))}.
\end{equation*}
\end{remark}

\subsection{Transformation into a canonical form}

Let $P=(a,b)$ $\in$ $\R_+^2$ and the function $h_P$ of class $H_{\Lbar}$
 \beq 
 h_P(L)=a\Big(\arctan(b(L-\Lbar))+\frac{\pi}{2}\Big).
 \label{eq:hP}
 \eeq
We use the notation $k:=h_P(\Lbar)=a\frac{\pi}{2}$.
Let $P : \gamma \mapsto P(\gamma)=(a_0, b_0+\gamma)$
where $(a_0,b_0) \in \R^{*^2}_+$.
Then we can associate $P(\gamma)$ to a new system $(S_{\gamma}(a_0,b_0))$ obtained from \eqref{eq:S}
\begin{equation}
  \left\{ \tag{$S_{\gamma}(a_0,b_0)$}
      \begin{aligned}
\dot{E} &=  b_E L - d_E E -h_{P(\gamma)}(L) E,\\ 
\dot{L} &= h_{P(\gamma)}(L) E - d_L L - c L^2.
\end{aligned}
    \right.
    \label{eq:S2}
\end{equation}
This system has a positive equilibrium $(\Ebar,\Lbar)$ and the Jacobian matrix of the system evaluated in $(\Ebar,\Lbar)$ is:
$$J_{P(\gamma)}=
\left( \begin{array}{cc}
-d_E - h_{P(\gamma)}(\Lbar)   &  b_E - h'_{P(\gamma)}(\Lbar) \Ebar  \\[12pt]
h_{P(\gamma)}(\Lbar) & h'_{P(\gamma)}(\Lbar) \Ebar - d_L - 2c\Lbar 
\end{array}\right)
,$$
We set $\lambda_{1,2}(\gamma)=\alpha(\gamma) \pm i\beta(\gamma)$ the eigenvalues of $J_{P(\gamma)}$, when the discriminant of the characteristic polynomial of $J_{P(\gamma)}$ is negative.

\subsection{Main result}

Using function $T$ from \eqref{bkgk}, we define
\beq
 b(a) := \f{T(a)}{a}, \quad a_{crit}:=\frac{2k^+}{\pi} > 0.
 \label{eq:baacrit}
\eeq

\begin{theorem}\label{mainres}
There exists $\tilde{a}>0$ such that:
If $a>max(\tilde{a},a_{crit})$, $(S_{\gamma}(a,b(a)))$ has a supercritical Hopf Bifurcation in $\gamma=0$.
In particular:
\begin{enumerate}
\item there exists $\gamma_1$<0 such that for all $\gamma \in(\gamma_1,0]$, $(\Ebar,\Lbar)$ is a stable focus,
\item for all $U$ neighborhood of $(\Ebar,\Lbar)$, there exists $\gamma_2>0$ such that for all $\gamma \in [0, \gamma_2)$, $(\Ebar,\Lbar)$ is an unstable focus surrounded by a stable limit cycle contained in $U$, which has an amplitude that grows when $\gamma$ grows.
\end{enumerate}
\end{theorem}
\begin{remark}
$k_+$ is given by \eqref{eq:k+}, and $\tilde{a}$ is such that the normal form coefficient $\alpha_N$ (see \cite{forme}) of our system is negative if $a > \tilde{a}$. We simply give a numerical justification of the existence of $\tilde{a}$ as the computations appear to be very long (see the proof below).
\end{remark}
\begin{remark}
 The value of $a$ must be greater than $a_{crit}$ to ensure that the linearized operator has complex eigenvalues.
\end{remark}
The bifurcation diagram for $S_{\gamma}(a_0, b_0)$ in Figure \ref{figAUTO} is obtained by XPPAUT software~\cite{Ermentrout}.
\begin{figure}[h!]
 \center{\includegraphics[width=.5\linewidth]{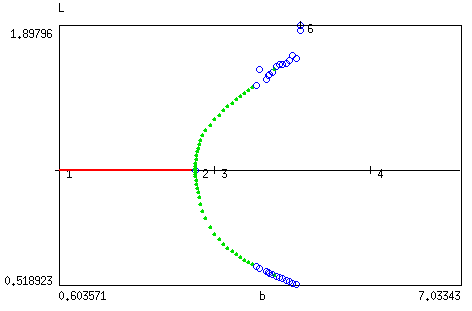}}
 \caption{Supercritical Hopf bifurcation diagram with $a_0=0.2$. The bifurcation parameter $b$ is in $x$-axis, the diagram shows extreme values of the periodic solution for $L$ (the $L$ scale is in $y$-axis). The steady state is stable (red line) until the bifurcation point (point number 2) is reached. A periodic solution appears and is stable (green points) until a bigger value of $b$, where it becomes unstable (blue circles). The amplitude of the periodic solution grows with the parameter $b$.}
 \label{figAUTO}
\end{figure}

\begin{proof}[Proof of Theorem \ref{mainres}]
 We set $\lambda_{1,2}(\gamma)=\alpha(\gamma)\pm i\beta(\gamma)$ (with $\gamma$ a real parameter), the two eigenvalues of $J_{P(\gamma)}$ the Jacobian matrix associated to our system and computed in (0,0). We call $\gamma_c$ a bifurcation value, and $\alpha_N(\gamma)$ the normal form coefficient of the system (see \cite{forme}).

Firstly, we only need to study complex conjugate and pure imaginary eigenvalues of $J_{P(\gamma)}$ to find the bifurcation value $\gamma_c$, which means also to look for $\gamma_c$ such that $\alpha (\gamma_c)=0$ and $\beta(\gamma_c) \neq 0$. Thanks to Proposition \ref{prop:eigenvalues} we know that this is the case when $k>k_+$ {\it i.e.} $\f{a\pi}{2}>k_+$ or equivalently $a>a_{crit}$ (by definition, $a_{crit}=\f{2k_+}{\pi}$). Moreover since $h'(\Lbar)=ab$ (direct computation from \eqref{eq:hP}), we know that the bifurcation value is located at the level of the graph $G$ of function $b$ (defined in \eqref{eq:baacrit})
\beq\label{GG}
G:= \{(a,b) \in \R^2,\ a>a_{crit},\ T(a) = a b = h'_{P(\gamma)}(\Lbar)\}.
\eeq
And we can set $\gamma_c=0$.

Secondly, we have to see if $\frac{d\alpha}{d\gamma}(\gamma_c)>0$, this means to check that $\tr(J_{P(\gamma)})$ changes sign at the bifurcation value $\gamma_c$. Let $\gamma \longmapsto z(\gamma)=\alpha(a_0, b(a_0)+\gamma)$. We recall that $\alpha(\gamma)$ is a function of $a$ and $b$.

Since
\[
\alpha = \frac{\tr(J_{P(\gamma)})}{2}=\frac{1}{2}\Big(-d_E-\frac{a\pi}{2}+ab\Ebar-d_L-2c\Lbar \Big),
\]
we have $z'(\gamma) = \partial_b\alpha=\frac{a\Ebar}{2}$ and we obtain that $z'(\gamma_c)=z'(0)=\frac{a\Ebar}{2}$ and it is always positive.

Thirdly, we have to study the normal form coefficient of the system computed in $\gamma_c=0$ and find when $\alpha_N(\gamma_c)\neq0$. To get the normal form coefficient, we have to transform the system \eqref{eq:S2} and we use the steps from \cite{forme}.
In a first step we reduce the initial system \eqref{eq:S2} to a system where the equilibrium $(\Ebar,\Lbar)$ becomes the origin. 
By the change of variables $x=E-\Ebar$ and $y=L-\Lbar$, $(S_{\gamma}(a_0,b_0))$ becomes:
\begin{equation}
  \left\{
      \begin{aligned}
\dot{x} &=b_E(y+\Lbar)-d_E(x+\Ebar)-a\Big(\arctan(by)+\frac{\pi}{2}\Big)(x+\Ebar),  \\
\dot{y} &=a\Big(\arctan(by)+\frac{\pi}{2}\Big)(x+\Ebar)-d_L(y+\Lbar)-c(y+\Lbar)^2 . 
	 \end{aligned}
    \right.
    \label{eq:chgt}
\end{equation}
Then as $(\Ebar,\Lbar)$ is an equilibrium, we can simplify \eqref{eq:chgt} into
\begin{equation}
  \left\{
      \begin{aligned}
\dot{x} &=b_Ey-d_Ex-\frac{a\pi}{2}x-a\Big(\arctan(by)\Big)(x+\Ebar),  \\
\dot{y} &=a\Big(\arctan(by)+\frac{\pi}{2}\Big)(x+\Ebar)+\frac{a\pi}{2}x-d_Ly-cy^2-2cy\Lbar,
	 \end{aligned}
    \right.
    \label{eq:chgt2}
\end{equation}
which we write as
\begin{equation}
  \left\{
      \begin{aligned}
\dot{x} &=b_Ey-d_Ex-\frac{a\pi}{2}x-aby\Ebar +f(x,y),  \\
\dot{y} &=\frac{a\pi}{2}x-d_Ly-2cy\Lbar+aby\Ebar +g(x,y),
	 \end{aligned}
    \right.
    \label{eq:chgt3}
\end{equation}
where
\[
f(x,y) =aby\Ebar-a\arctan(by)(x+\Ebar), \quad
g(x,y) =-aby\Ebar+a\arctan(by)(x+\Ebar)-cy^2.
\]

The system \eqref{eq:chgt3} can also be written under the matrix form
$$\left( \begin{array}{c}
\dot{x}    \\[12pt]
\dot{y}  
\end{array}\right)=
\left( \begin{array}{cc}
-\frac{a\pi}{2}-d_E   &  b_E - ab\Ebar  \\[12pt]
\frac{a\pi}{2} & - d_L - 2c\Lbar +ab\Ebar
\end{array}\right)
\left( \begin{array}{c}
x \\[12pt]
y
\end{array}\right)+
\left( \begin{array}{c}
f(x,y) \\[12pt]
g(x,y)
\end{array}\right).$$
We call $M$ the first ($2 \times 2$) matrix in the right-hand-side.

Now, to obtain the normal form coefficient, one way is to perform a linear change of variables so as to get
\beq
\left( \begin{array}{c}
\dot{X}    \\[12pt]
\dot{Y}  
\end{array}\right)=
N
\left( \begin{array}{c}
X \\[12pt]
Y
\end{array}\right)+
\left( \begin{array}{c}
F(X,Y) \\[12pt]
G(X,Y)
\end{array}\right), \quad 
N := \left( \begin{array}{cc}
0   & -\omega  \\[12pt]
\omega & 0
\end{array}\right).
\label{sys:matrices}
\eeq
In our case, we can have an idea of the normal coefficient only in a neighborhood of $\gamma=0$. Because we want to make a simple linear change of variables, we are looking for a matrix $P$ such that $PMP^{-1}=N$ and that \emph{at the bifurcation value} $\gamma=0$, $\tr(M)=0=\tr(N)$ and $\det(M)=-A^2-BC=\omega^2>0$.

We set
$M=
\left( \begin{array}{cc}
A  & B  \\[12pt]
C & -A
\end{array}\right)
$ and we can choose 
$P=
\left( \begin{array}{cc}
\frac{\omega+A}{2B\omega}  & \frac{1}{2\omega}  \\[12pt]
\frac{\omega-A}{2B\omega} & -\frac{1}{2\omega}
\end{array}\right)
$,
$P^{-1}=
\left( \begin{array}{cc}
B  & B  \\[12pt]
\omega-A & -A-\omega
\end{array}\right)
.$

Next we obtain the matrix system \eqref{sys:matrices} where
\[
\left( \begin{array}{c}
X  \\[12pt]
Y 
\end{array}\right)=
P\left( \begin{array}{c}
x \\[12pt]
y
\end{array}\right)=
\left( \begin{array}{c}
\frac{x(\omega+A)}{2B\omega}+\frac{y}{2\omega} \\[12pt]
\frac{x(\omega-A)}{2B\omega}-\frac{y}{2\omega}
\end{array}\right),
\quad
\left( \begin{array}{c}
x  \\[12pt]
y 
\end{array}\right)=
P^{-1}\left( \begin{array}{c}
X \\[12pt]
Y
\end{array}\right)=
\left( \begin{array}{c}
(X+Y)B \\[12pt]
(\omega-A)X+(-\omega-A)Y
\end{array}\right),
\]
$$\left( \begin{array}{c}
F(X,Y)  \\[12pt]
G(X,Y) 
\end{array}\right)=
P\left( \begin{array}{c}
f(x,y) \\[12pt]
g(x,y)
\end{array}\right)=
\left( \begin{array}{c}
\frac{\omega+A}{2B\omega}f(x,y)+\frac{1}{2\omega}g(x,y)\\[12pt]
\frac{\omega-A}{2B\omega}f(x,y)-\frac{1}{2\omega}g(x,y)
\end{array}\right)=
\left( \begin{array}{c}
f(x,y)\Big(\frac{\omega+A}{2B\omega}-\frac{1}{2\omega}\Big)-\frac{1}{2\omega}cy^2 \\[12pt]
f(x,y)\Big(\frac{\omega-A}{2B\omega}+\frac{1}{2\omega}\Big)+\frac{1}{2\omega}cy^2
\end{array}\right).$$

In a final step we compute the normal form coefficient using the previous formulas and the expression that exists in two dimensions given in \cite{forme} which is:
\begin{eqnarray*}
\alpha_N(\gamma=0) &=& \frac{1}{16}\Big(F_{XXX}+F_{XYY}+G_{XXY}+G_{YYY}\Big)\\
&\ &-\frac{1}{16\omega}\Big(G_{XY}(G_{XX}+G_{YY})-F_{XY}(F_{XX}+F_{YY}) + F_{XX}G_{XX}-F_{YY}G_{YY}\Big).
\end{eqnarray*}
The coefficient is easy but very tedious to compute, and we used the computer algebra system Maple~\cite{maple} to get its expression.

In our case the coefficient is equal to zero for some value $\tilde{a}>0$, and is always negative for $a>\tilde{a}$ (as it appears that $\tilde{a}<a_{crit}$, this is sufficient by definition of \eqref{GG}).
Then $\alpha_N(\gamma_c)\neq0$ for $a\neq\tilde{a}$.

Finally, we want to have for all real $\gamma$ in a neighborhood of 0, $\alpha_N(\gamma)\alpha(\gamma)<0$.
Thanks to Maple we have $\alpha_N(0)<0$, in a neighborhood of $\gamma=0$, for $a>\tilde{a}$ with $\tilde{a}$ small.

So we can apply the Hopf bifurcation theorem that ensures there exists a limit cycle (periodic solution) when $\alpha(\gamma)>0$ (i.e $\tr(J_P(\gamma))>0$), and moreover this cycle is stable as $\alpha(\gamma)>0$: we are faced to a supercritical bifurcation.
\end{proof}

\subsection{Discussion on the period of the oscillations}

\label{s4}
The period of the oscillating solutions are relevant to the biological problem in consideration, because they can be compared with observations in nature.

\begin{proposition}
As $\gamma \to 0^+$, the periodic solution of the system $(S_\gamma(a,b(a)))$ has a frequency $\omega$ and a period $T_0 = 2 \pi / \omega$ given by the expression
\[
\omega =\frac{1}{\sqrt{d_E+k}}\Big[k^2(b_E-d_L-d_E)+k(-2{d_E}^2-b_Ed_E-d_Ed_L) -{d_E}^3\Big]^{\frac{1}{2}}.
\]
\end{proposition}
\begin{proof}
As $\gamma \to 0^+$, the oscillations frequency is given by the imaginary part of the root of the polynomial equation \eqref{eq:lambda} in the case of non-trivial steady state. The frequency is $\omega_{\gamma}=\sqrt{\det(J_{P(\gamma)})}$, where the expression of $\det(J_{P(\gamma)})$  is
\begin{eqnarray*}
\det(J_{P(\gamma)})&=&\frac{1}{d_E+k}\Big[k^2(b_E-d_L-d_E)+k(-2{d_E}^2-b_Ed_E-d_Ed_L) -{d_E}^3\Big] .
\end{eqnarray*}
Then the expression of $\omega = \omega_0$ follows.
\end{proof}
\begin{remark}
At the bifurcation value, the parameter $k$ can be linked with the period $T_0$. Let $T_0$ a given period observed experimentally, then we find a corresponding value for $k$ as the positive root of the following characteristic polynomial:
$$k^2\Big[T_0^2(b_E-d_L-d_E)\Big]+k\Big[T_0^2(-2d_E^2-b_Ed_E-d_Ed_L)-4\pi^2\Big]-T_0^2d_E^3-4\pi^2d_E .$$

Away from the bifurcation value, the real part of the eigenvalues is greater than zero and the period of the oscillations can only be obtained numerically. Unfortunately, this case is more relevant as the Hopf bifurcation theorem asserts that the amplitude is increasing with the parameter $\gamma$.
In other words, for fixed $a$ the amplitude of the oscillations is an increasing function of $b$.
\end{remark}

\section{Conclusion}

We show that introducing internal regulation in the form of a larval-density-mediated hatching rate in a compartmental model for mosquito population dynamics induces stable oscillations. These oscillations can be rather simply understood from the mathematical point of view either as cycles produced by a Hopf bifurcation (Theorem \ref{mainres}), in a first parameter regime, or as the typical slow-fast behavior (close to FitzHugh-Nagumo model, Theorem \ref{thm:mainslowfast}) in a second parameter regime.

Our study supports the idea that understanding internal life-cycle regulation can effectively help modeling and simulating population dynamics properly. Ongoing experiments of some of the authors try to reproduce the larval density impact on hatching which was observed in \cite{art:hatch} and may shed some light on this misunderstood phenomenon. In particular, restricting the parameters and possible oscillations range could only be reached by assessing as precisely as possible the actual hatching feedback.

In this paper we neglect environmental variations. Therefore it leaves open for future studies the deep question of linking internal life-cycle regulation and external variations (induced, for instance, by rainfall and temperature) in order to get a better description of the mosquito populations dynamics. However, it was observed that population oscillations may happen on periods much shorter than seasonal variations, and this justifies the study of internal regulations as possible triggers.

Another possible extension of our works is the adaptive dynamics of hatching regulation trait. Indeed, synchronizing the egg hatching may be beneficial for a population in a given environment, but also be detrimental if rare and extreme events can annihilate larval population, for instance. The egg stage can be seen indeed as a quiescent, refuge state for the species (this approach was studied in \cite{Yan.Assessing}). Here we prove that positive feedback of larvae on egg hatching tends to make the population size oscillate, creating distinct generations ({\it synchronizing effect}) while negative feedback tends to stabilize the population size, which may be detrimental on the long run if, for example, the favorable period for larvae and adult development is typically short.

 
\paragraph{Acknowledgements.}
BP has received funding from the European Research Council (ERC) under the European Union's Horizon 2020 research and innovation programme (grant agreement No 740623).
MS, NV and DAM acknowledge partial funding from Inria, France and CAPES, Brazil
(processo 99999.007551/2015-00), in the framework of the STIC AmSud project MOSTICAW and from CAPES/COFECUB project Ma-833 15 “Modeling innovative control method for Dengue fever”.
MS and NV acknowledge partial funding from the ANR blanche project Kibord: ANR-13-BS01-0004
funded by the French Ministry of Research.

\appendix

\section{Observations on a class of hatching functions}
\label{app:Hill}

Among the many possible choices for a S-shaped hatching function $h$, we numerically and theoretically explore the typical family of Hill functions. 
We assume the following form with parameters $a, \lambda, p > 0$
\beq
	h(L) = h_m + a \f{L^p}{\lambda^p + L^p}, \quad h_m > \df{d_E d_L}{b_E - d_L}.
	\label{eq:Hillshape}
\eeq

Steady states ($\Ebar, \Lbar$) of \eqref{eq:S} are such that $\Lbar$ is a solution of $Q(L) = 0$, where
\begin{multline*}
 Q(L) = -c  L^{p+1} \big(h_m + a + d_E \big) + L^p \big( (h_m + a)(b_E - d_L) - d_E d_L \big) \\ - c \lambda^p L ( h_m + d_E ) + \lambda^p \big( h_m (b_E - d_L) - d_E d_L \big).
\end{multline*}

The following lemma is a straightforward consequence of this computation
\begin{lemma}
When $p=1$ and $h$ is of type \eqref{eq:Hillshape}, there is a unique steady state of \eqref{eq:S}.
\end{lemma}
When $p=1$ and $h$ is of type \eqref{eq:Hillshape}, then $h'' < 0$, a property that is lost when $p>1$.
Therefore, to simplify the choice of the parameters, now we assume 
\beq
d_E= 0.
\label{immortaleggs}
\eeq
Then, condition \eqref{eq:sufcond} is fulfilled, the steady state of \eqref{eq:S} is unique and is given by
\[
	\Lbar = \f{b_E - d_L}{c}, \quad \Ebar = \f{b_E \Lbar}{h(\Lbar)}.
\]

\begin{proposition}
 Let $h$ be of type \eqref{eq:Hillshape} and assume condition \eqref{immortaleggs} holds.
 Then \eqref{eq:S} has a unique positive steady state $(\Ebar, \Lbar)$ and its linearization has eigenvalues with negative real parts if and only if 
 \beq
 k > \df{a}{1+\alpha^p} > \df{\alpha^p+1}{p \alpha^p} k \big(2 + \f{k-d_L}{b_E} \big), \quad \alpha = \df{\lambda}{\Lbar}, \quad k = h(\Lbar).
 \label{cond:k}
 \eeq
\end{proposition}
\begin{proof}
Necessarily $k > k_+ = 0$ (where $k_+$ is defined in \eqref{eq:k+}). Hence the eigenvalues of the linearized system at $(\Ebar, \Lbar)$ are in $\mathbb{C} \backslash \R$ and the condition for instability of the steady state from Proposition \ref{prop:eigenvalues} simply reads $h'(\Lbar) > T( h(\Lbar))$, where $T$ is defined in \eqref{bkgk}.
Then we compute
\[
	h(\Lbar) = h_m + \f{a}{1 + \alpha^p}, \quad h'(\Lbar) = \f{a p \alpha^p}{\Lbar \big( \alpha^p + 1 \big)^2}.
\]
The right-hand side inequality in \eqref{cond:k} comes from $h'(\Lbar) > T(k)$ and the left-hand side from $h_m > 0$.
\end{proof}
 
If all parameters but $\alpha$ and $a$ are fixed, then condition \eqref{cond:k} can be fulfilled if and only if
\beq
	k < (p - 2) b_E + d_L.
	\label{cond:alpha}
\eeq
Indeed, we need to find $\alpha > 0$ such that $2 + \frac{k-d_L}{b_E} < p \frac{\alpha^p}{1 + \alpha^p}$.
Note that in particular, this is impossible when $p \leq 1$ (since $c \Lbar = b_E - d_L > 0$ by hypothesis).

We provide below numerical results showing consistent oscillations under condition \eqref{cond:alpha}, for $2$- and $3$-dimensional systems \eqref{eq:S} and \eqref{eq:S3}.
To explore the possible behaviors depending on the function $h$ of type \eqref{eq:Hillshape}, we fix the biological parameters (including $\Lbar$), $p > 1$ and $k = h(\Lbar) > 0$ such that \eqref{cond:alpha} holds.
We introduce the notation $X(k) := \f{1}{p} (2 + \f{k-d_L}{b_E}) < 1$ and use two parameters: $\iota \in (0, 1- X(k))$ and $\zeta \in (0, 1)$, in order to represent the full range of \eqref{cond:k}. More precisely, we will parametrize $a$ and $\alpha$ with $\iota, \zeta$, as functions of $k$, and then we can go back to a function in \eqref{eq:Hillshape} by letting $\lambda = \alpha \Lbar$ and $h_m = k - a / (1 + \alpha^p)$.

\begin{figure}
[h!]
\includegraphics[width=.5\textwidth]{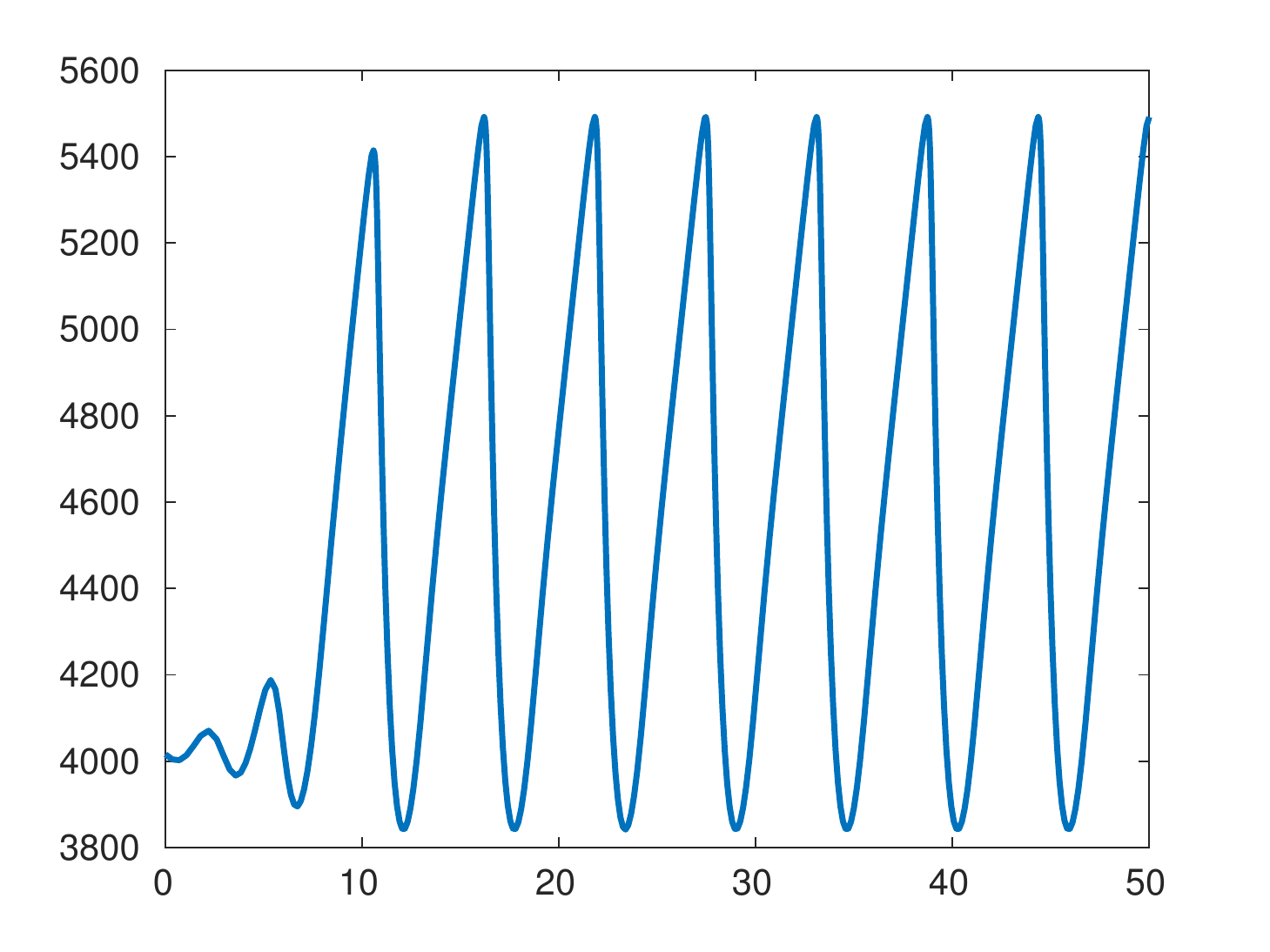}
\includegraphics[width=.5\textwidth]{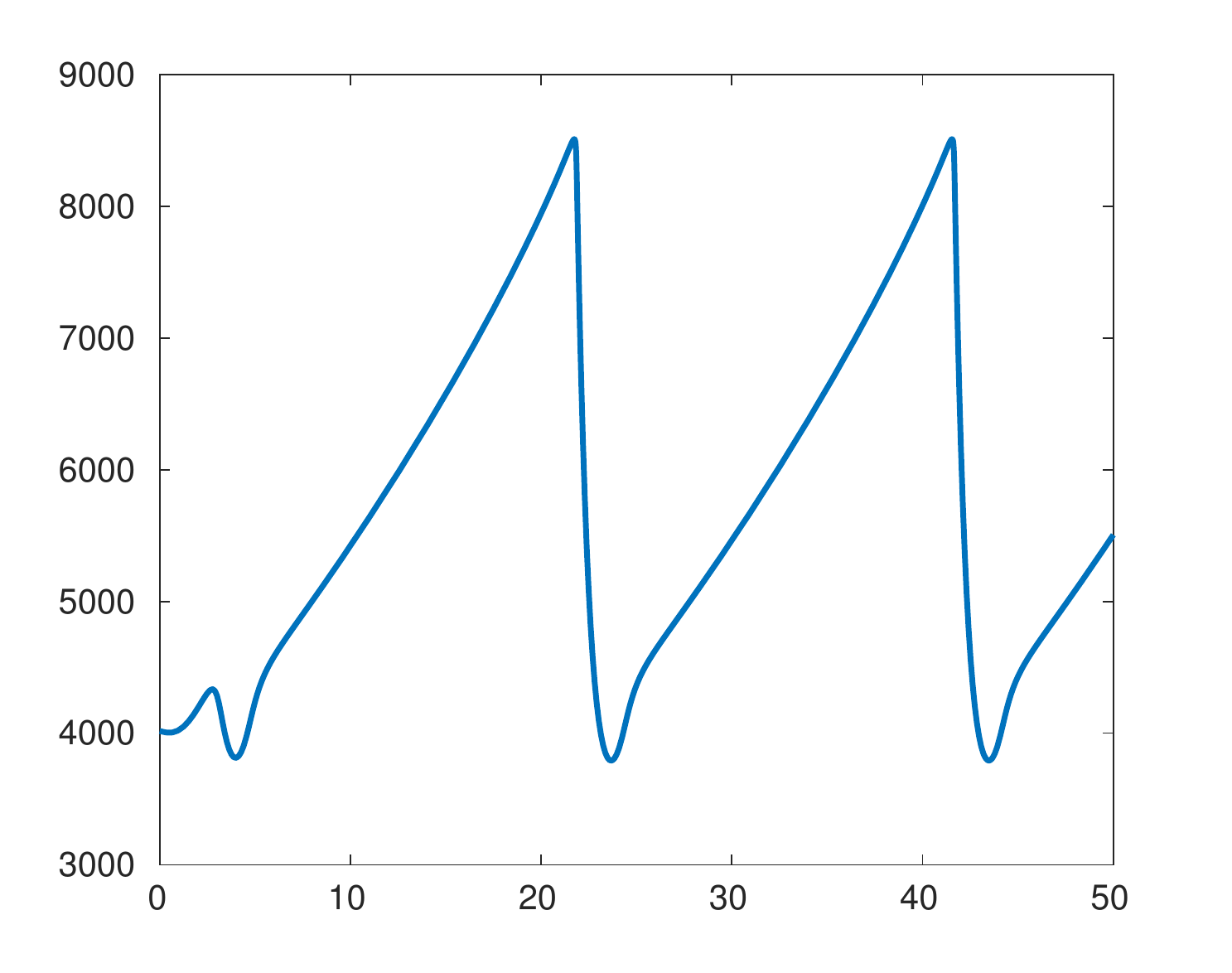}
\includegraphics[width=.5\textwidth]{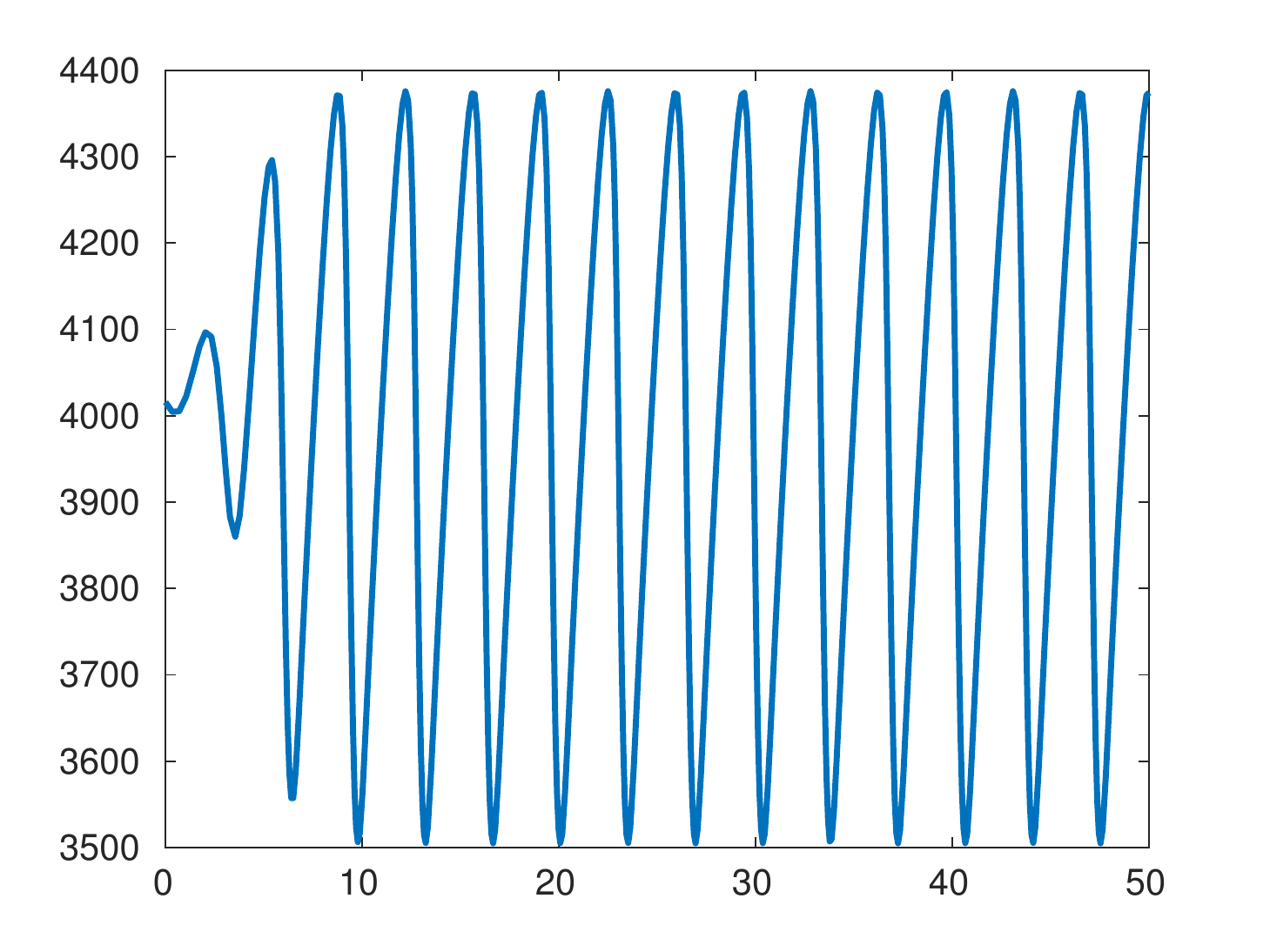}
\includegraphics[width=.5\textwidth]{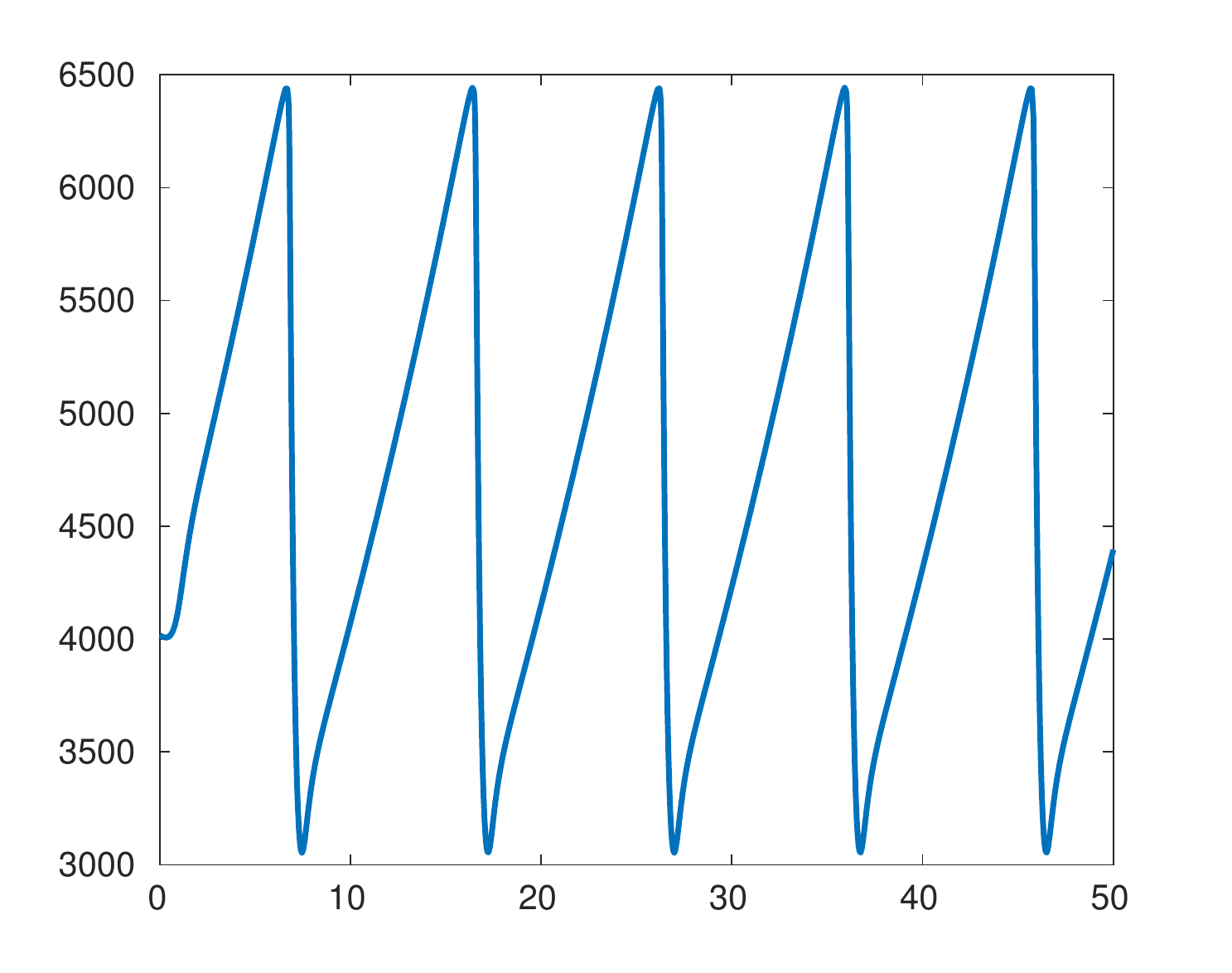}
\caption{Egg dynamics from \eqref{eq:S} for $h$ of Hill function type.
All parameters being fixed, including $p=3$ and $k=0.5$, $\iota = 0.05$ (top) or $\iota = 0.2$ (bottom) and $\zeta = 0.2$ (left) or $\zeta=0.8$ (right).}
\label{fig:Hill}
\end{figure}

\begin{figure}[h!]
\includegraphics[width=.5\textwidth]{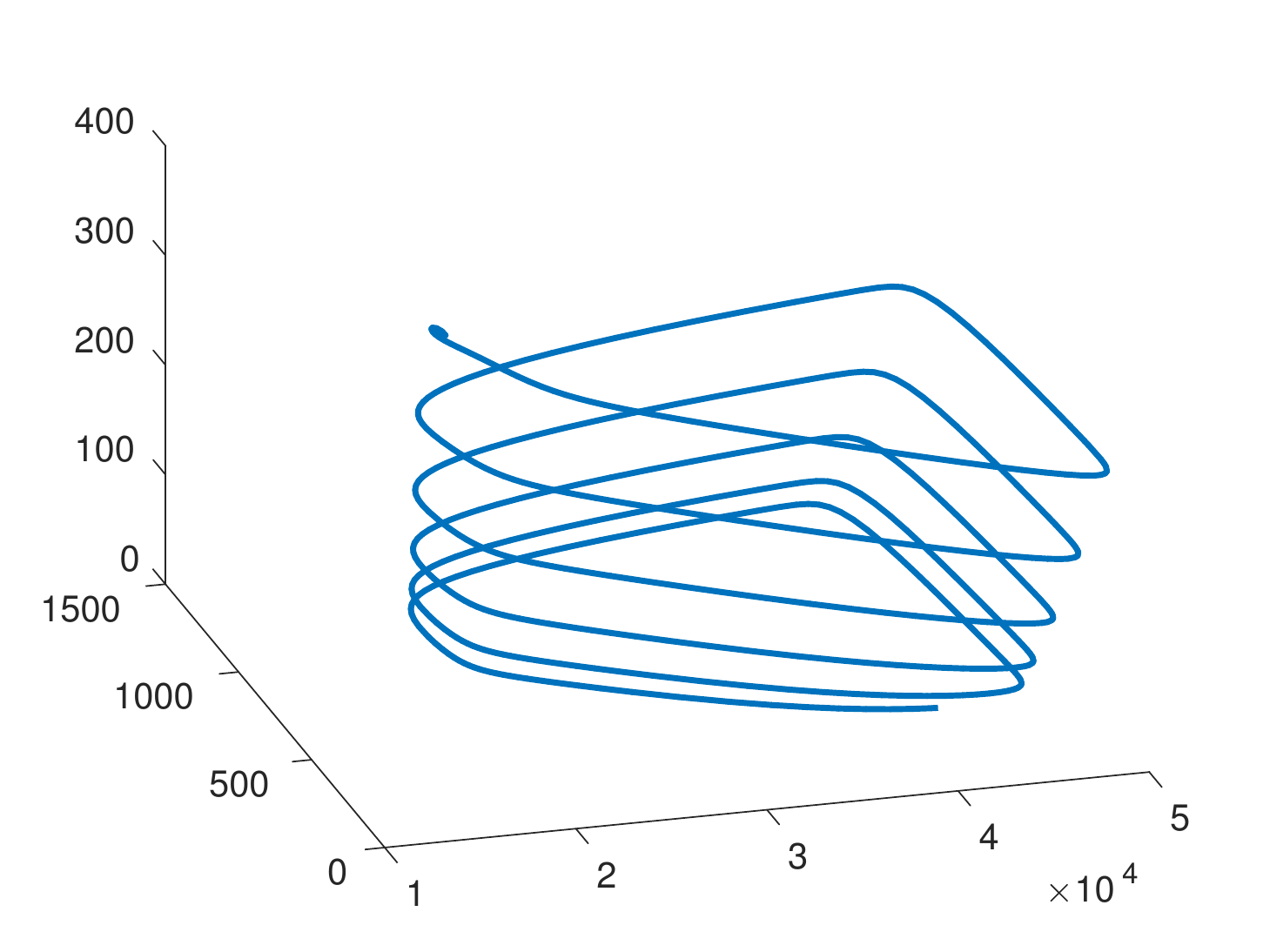}
\includegraphics[width=.5\textwidth]{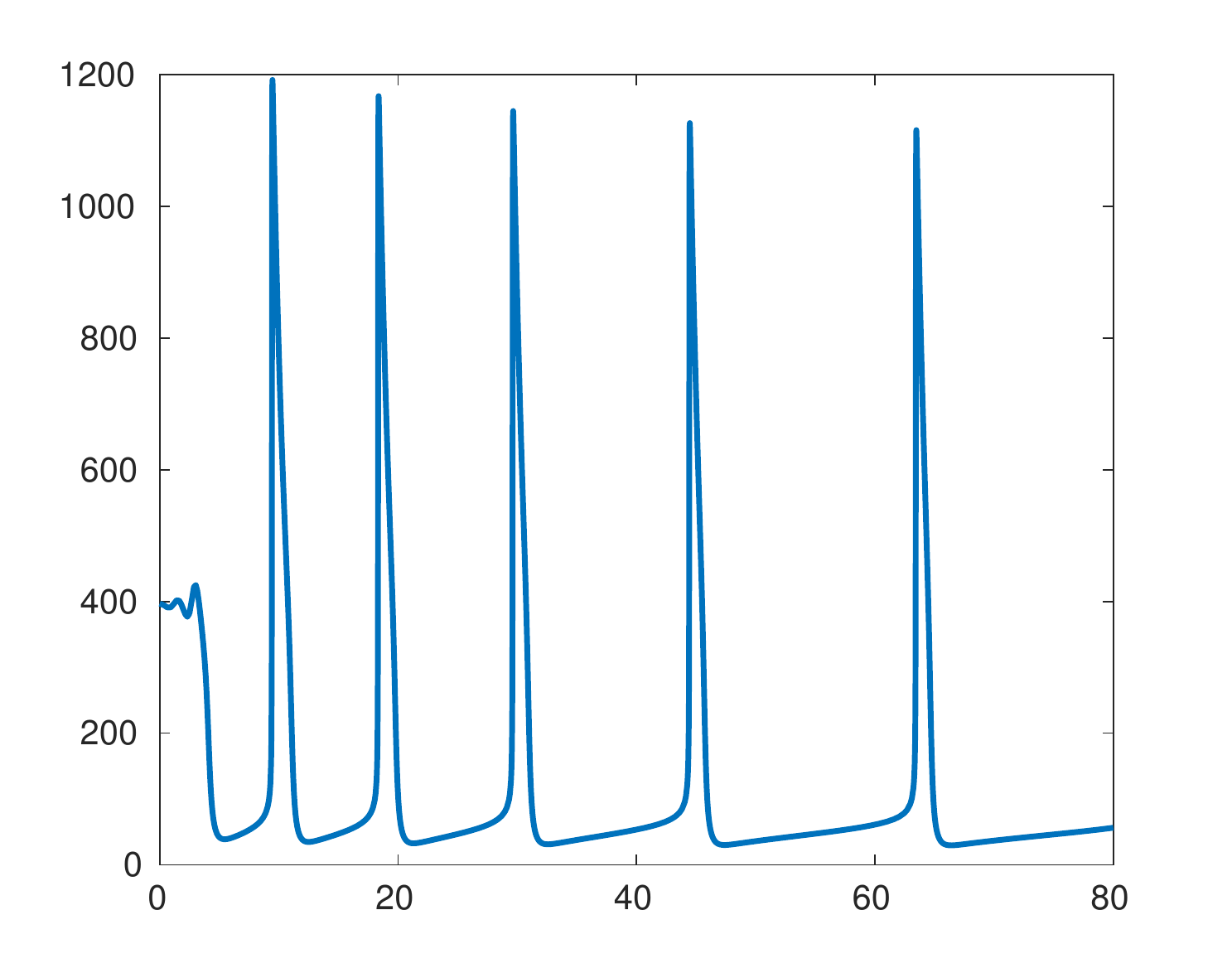}
\includegraphics[width=.5\textwidth]{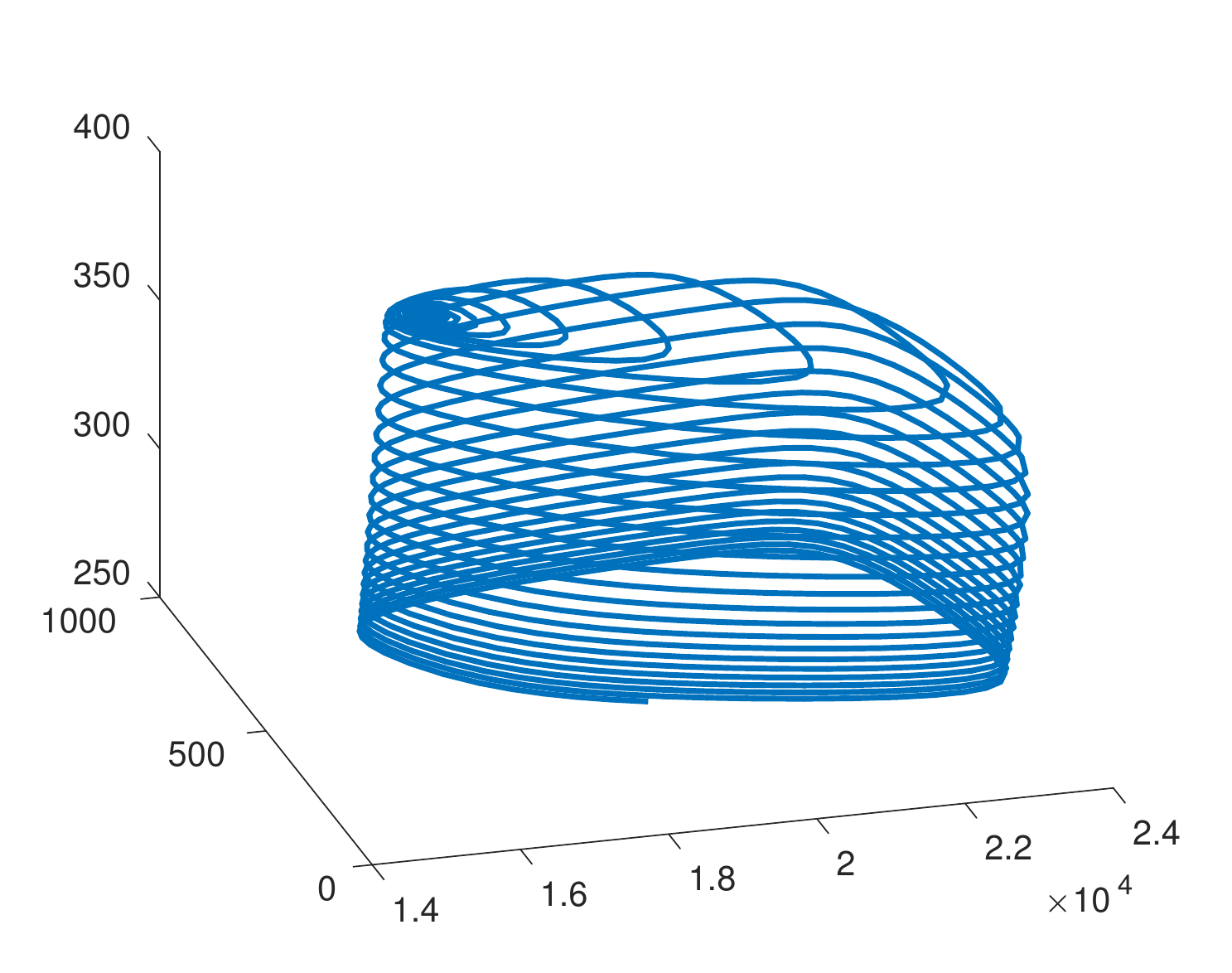}
\includegraphics[width=.5\textwidth]{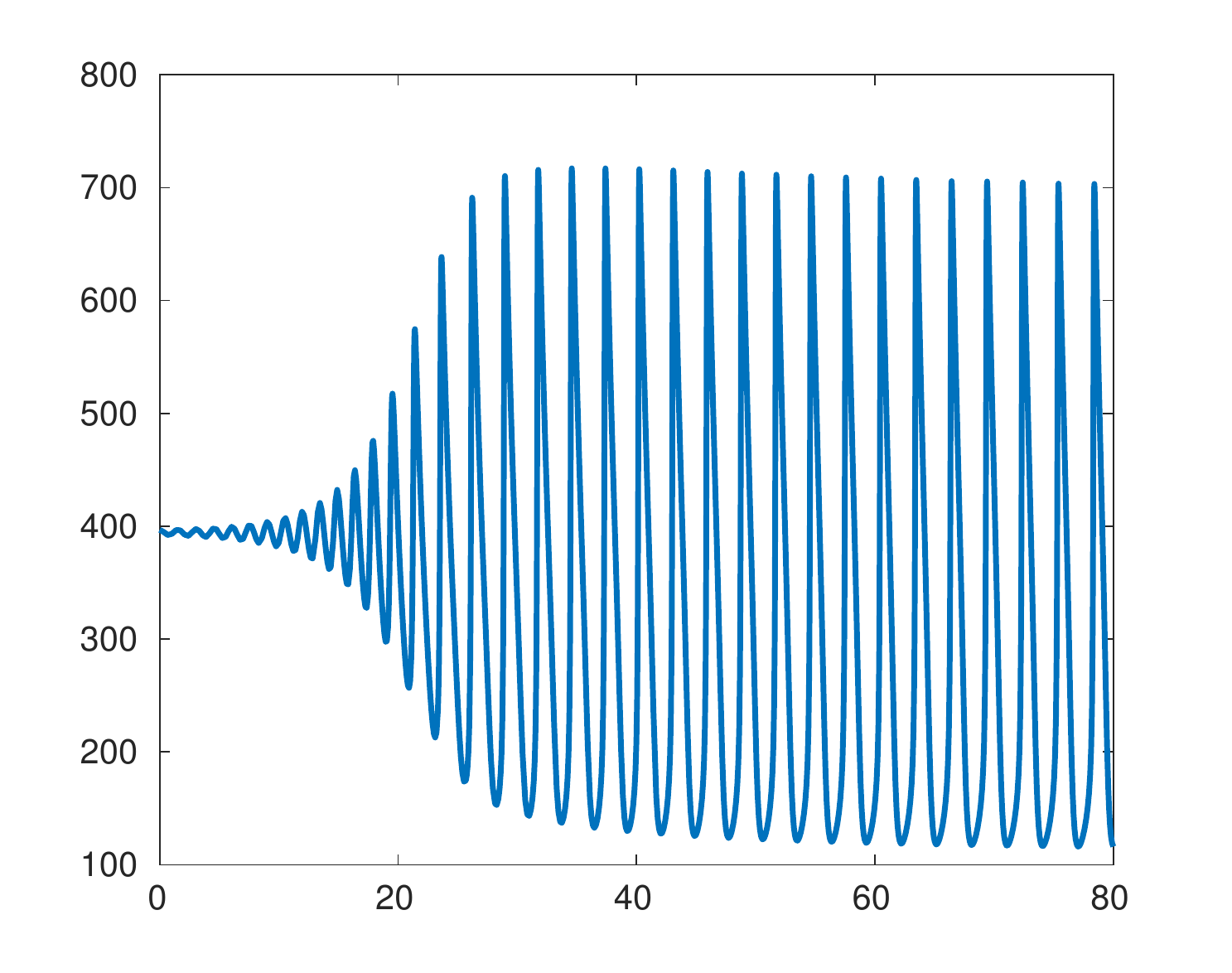}
\caption{Numerical solutions of \eqref{eq:S3} for $h$ defined by two different Hill functions.
All parameters being fixed, including $p=3$, $k=0.5$, $\iota = \f{1-X(k)}{10}$ and choosing $\zeta = 0.1$ (top) or $\zeta = 0.9$ (bottom).}
\label{fig:HillELA2}
\end{figure}

We choose
\[
 \alpha_{\iota} (k) = \Big( \f{X(k) + \iota}{1 - \big( X(k) + \iota \big)} \Big)^{1/p}
\]
and
\[
 a_{\zeta, \iota} (k) = \zeta k (1 + \alpha_{\iota}^p) + (1 - \zeta) k \f{(1+\alpha_{\iota}^p)^2}{\alpha_{\iota}^p} X(k) = k \df{1 - \zeta (1 - (X(k) + \iota)) }{(X(k) + \iota)(1 - (X(k) + \iota))}.
\]
For any choice of $\iota$ and $\zeta$, we end up with system \eqref{eq:S}, $h$ given by \eqref{eq:Hillshape}, featuring a unique, (locally linearly) unstable positive steady state.
At least numerically, solutions always exhibit periodic oscillations, as can be seen in Figure \ref{fig:Hill} for egg dynamics.

The above computations extend to the $3$-dimensional system \eqref{eq:S3}, and numerical observations are similar. Indeed, the condition \eqref{cond:alpha} guaranteeing positivity of the trace of the Jacobian at the unique positive equilibrium, rewrites for system \eqref{eq:S3} as $(p-2) \f{\beta_E \tau_L}{\delta_A} + \delta_L + \tau_L - \delta_A > k.$
In this case we define
\[
 X(k) := \df{\delta_A}{p \beta_E \tau_L} \big( 2 \f{\beta_E}{\delta_A}\tau_L - \delta_L - \tau_L + \delta_A + k  \big),
\]
and the above condition is equivalent to $X(k) < 1$.

Exactly as in the two-dimensional case, we explore the full range of \eqref{cond:k} by choosing the parameters $(\iota, \zeta) \in (0, 1 - X(k)) \times (0, 1)$ and defining $\alpha_{\iota} (k)$ and $a_{\iota, \zeta} (k)$ by the same formulas as before.
For all the numerical values we took for $\iota$ and $\zeta$, we always found oscillating solutions. Examples (dynamics of larvae and of $(E, L, A)$ in the three dimensional space) are shown in Figure \ref{fig:HillELA2}.

\section{Amplitude and period computation in the slow-fast regime}
\label{app:slowfast}

In the slow-fast approach, system \eqref{eq:S} exhibits oscillations with known amplitude and period at the limit $\e \to 0$. 
We show here how to compute this amplitude analytically.
To do so, we simply compute the local extrema of $u \mapsto \f{\eta u^2}{h(\eta u)}$. The first-order necessary condition is $x h'(x) = 2 h(x)$, where $x = \eta u$.

This provides with a general method to determine the limit trajectories.
With the previous example from \eqref{eq:Hillshape}, $ h(x) = h_m + a \df{x^p}{\big( \alpha \Lbar \big)^p + x^p}$, this boils down to
\[
 2 (h_m + a) x^{2p} + (\alpha \Lbar)^p \big( (2 - p) a + 4 h_m \big) x^p + 2 h_m (\alpha \Lbar)^{2 p} = 0.
\]
Letting $y = x^p$, we end up with a second-order polynomial, for which the analytical computation can be pushed a few steps further.
In particular, its discriminant is
\begin{align*}
 \Delta &= (\alpha \Lbar)^{2p} \Big( \big( (2 -p) a + 4 h_m \big)^2 - 16 h_m (h_m + a) \Big) \\
 &= (\alpha \Lbar)^{2p} a \big( (2 -p)^2 a - 8 p h_m \big).
\end{align*}
Hence there are exactly two positive local extrema if and only if $ (2 -p)^2 a > 8 p h_m$ and  $(2 -p) a + 4 h_m < 0$.
The first condition implies the second one if $p > 2$, and the second one is impossible if $p \leq 2$. Therefore the only case when there are two local extrema is when 
\beq
  p > 2 \text{ and } \f{h_m}{a} < \f{(p-2)^2}{8 p}.
  \label{cond:twoLE}
\eeq
Under assumption \eqref{cond:twoLE} we find that the extrema ($y_M < y_m$) are located at
\[
  (\alpha \Lbar)^p \f{ (p-2) a - 4 h_m \pm \sqrt{ a^2 (p-2)^2 - 8 a p h_m} }{4 (h_m + a)}.
\]

Let $\xi_{\pm} = (p-2) a - 4 h_m \pm \sqrt{ a^2 (p-2)^2 - 8 a p h_m}$. With the notations of Lemma \ref{lem:trajectories},
\begin{align*}
 u_m &= \f{\alpha}{\eta} \Lbar \Big(\f{\xi_+ }{4 (h_m + a)} \Big)^{1/p}, \quad \phi_m = \f{\eta u_m^2}{h(\eta u_m)},
 \\
 u_M &= \f{\alpha}{\eta} \Lbar \Big(\f{\xi_- }{4 (h_m + a)} \Big)^{1/p}, \quad \phi_M = \f{\eta u_M^2}{h(\eta u_M)}.
\end{align*}

Then we can compute $u_r^0$ for $r \in \{m, M\}$ by solving $\f{\eta \cdot (u_r^0)^2}{h(\eta u_r^0)} = \phi_r$.
Unfortunately this cannot be done analytically. However, the amplitude of the oscillations in terms of $v$ is equal to
\[
 A_v := \phi_M - \phi_m.
\]

With $\Ebar = 1/\e$, we expect that the oscillations of $E$ have amplitude
\[
 \f{\phi_M - \phi_m}{\e} = \f{\alpha^2 \Lbar^2}{\eta \e}\Big(\f{\big(\f{ \xi_- }{4 (h_m + a)}\big)^{2/p}}{h_m + a \f{\xi_-}{h_m + \xi_-}} - \f{\big(\f{ \xi_+ }{4 (h_m + a)}\big)^{2/p}}{h_m + a \f{\xi_+}{h_m + \xi_+}} \Big),
\]
where $\eta = \f{\Lbar^2}{h(\Lbar)} = \f{\Lbar^2}{h_m + \f{a}{1 + \alpha^p}}$, by \eqref{eq:phieta}.
Hence the amplitude of egg oscillations is equal to
\[
 \f{1}{\e} A_v = \f{1}{\e} \f{\alpha^2 (h_m + \f{a}{1+\alpha^p})}{(4 (h_m + a))^{2/p}}\Big(\f{\xi_-^{2/p}}{h_m + a \f{\xi_-}{h_m + \xi_-}} -
 \f{\xi_+^{2/p}}{h_m + a \f{\xi_+}{h_m + \xi_+}} \Big).
\]
We can simplify this expression one step further by letting $\rho := h_m/a$.
 Then we notice that $q_{\pm} := \xi_{\pm}/a = p-2 - 4 \rho \pm \sqrt{(p-2)^2 - 8 p \rho}$ and deduce
\beq
 A_v = \f{\alpha^2}{1 + \alpha^p} \f{1 + \rho + \alpha^p}{(4(1+\rho))^{2/p}} \Big( \f{(\rho q_{-})^{2/p}}{1 + \f{\rho^2 q_-}{1 + \rho q_-}} - \f{(\rho q_{+})^{2/p}}{1 + \f{\rho^2 q_+}{1 + \rho q_+}} \Big).
\eeq

In particular we notice that the amplitude depends only on the function $h$ through $\rho$, $\alpha$ (hence $\Lbar$) and $p$, and not on any other biological parameter, under the constraints \eqref{cond:twoLE}.

An interesting case is when $p \to +\infty$, where $h$ approaches a step function from $h_m$ to $h_m + a$, with its jump located at $\alpha \Lbar$. In this limit we can compute the amplitudes in $u$ and $v$:
\[
\bepa
 A_u = \f{\alpha}{\Lbar} (h_m + a \mathds{1}_{\alpha < 1} + \f{a}{2} \delta_{\alpha=1}) \big( \sqrt{\f{\rho+1}{\rho}} - \sqrt{\f{\rho}{\rho+1}} \big),
\\[10pt]
A_v = \alpha^2 (h_m + a \mathds{1}_{\alpha <1} + \f{a}{2} \delta_{\alpha=1}) \f{1}{\rho (h_m + a)}.
\eepa
\]

If we assume $d_E = 0$ (for simplicity), using formula \eqref{formula:tau}, we can also obtain in this case an analytical expression for the period of the oscillations:
\[
 \tau = \frac{2}{h_m} \log \Big( \frac{h_m + a/2 - \alpha \Lbar}{h_m + a/2 - \alpha \Lbar \sqrt{\frac{\rho}{1+\rho}}} \Big) + \frac{2}{h_m + a} \log \Big( \frac{h_m + a/2 - \alpha \Lbar \sqrt{\frac{1+\rho}{\rho}}}{h_m + a/2 - \alpha \Lbar} \Big)
\]
Indeed, $h(u) = h_m$ if $u < \alpha \Lbar$ and $h(u) = h_m + a$ if $u > \alpha \Lbar$ so that $f(u, \phi(u)) = \eta_0 u \big( \xi - u \big)$ and $\phi'(u) = \frac{2 \eta_0}{h_m} u$ if $u < \alpha \Lbar$ and $\phi'(u) = \frac{2 \eta_0}{h_m + a} u$ if $u > \alpha \Lbar$.
\section{Numerical oscillations, period and amplitude close to the bifurcation}
\label{app:numPA}

We illustrate the statements from Section \ref{lll} with numerical examples. Biological parameters of \eqref{eq:S} are taken at a temperature around $25\degree{C}$ which leads to 
 $\Abar= 3.4$ mosquitoes per $100$ square meters (taken from a physical situation described in \cite{Density})  and  $b_E = 20.94, d_L = 0.15.$ (taken from \cite{Temp}).
To fit the condition $d_E \ll d_L$, $d_E$ is fixed arbitrarily at $\f{1}{180}$.
We note that condition \eqref{hyp:params} is satisfied: $b_E = 20.94 > 0.15 + \f{1}{180} = d_L + d_E$.

\begin{figure}[h]
 \includegraphics[width=.5\linewidth]{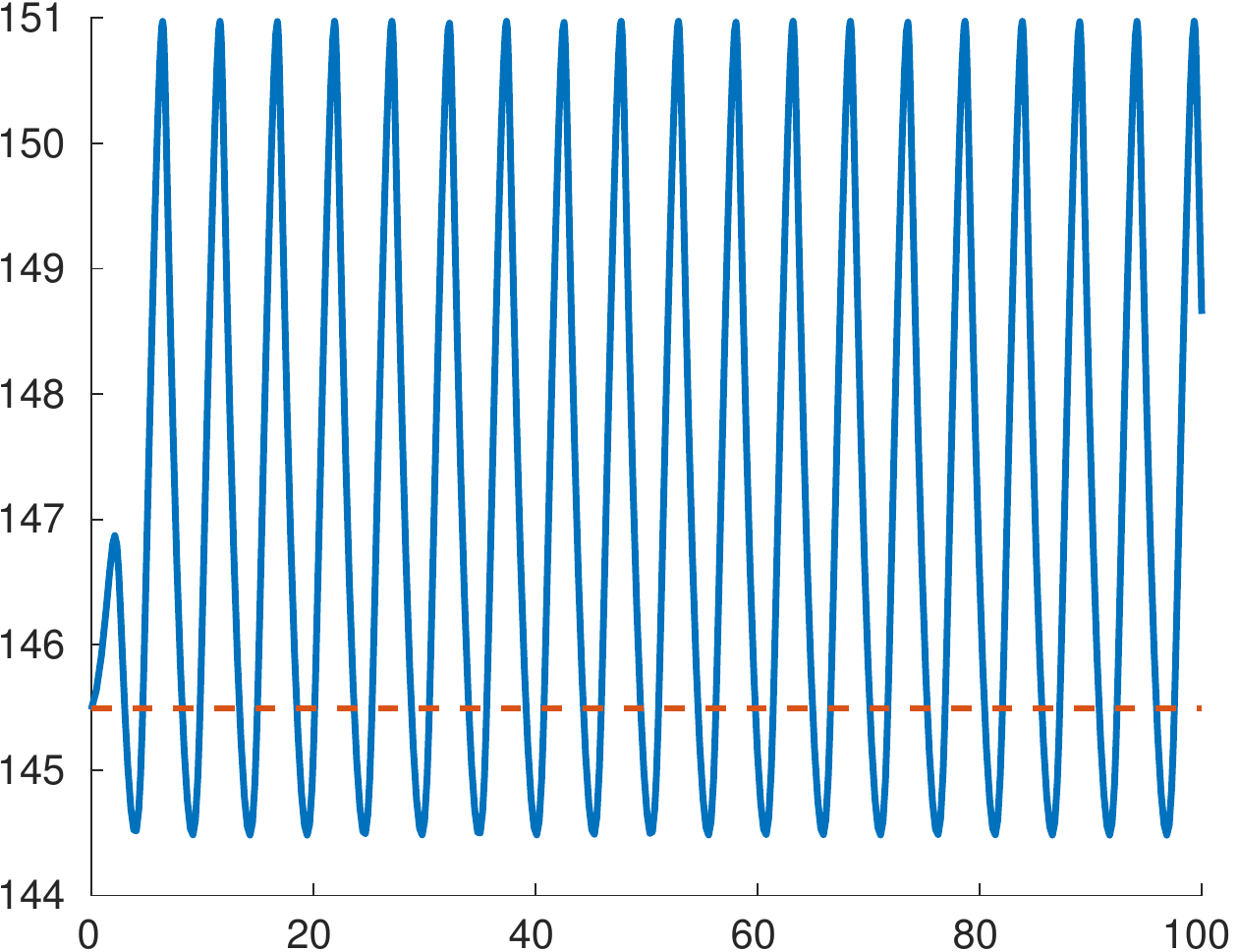}
 \includegraphics[width=.5\linewidth]{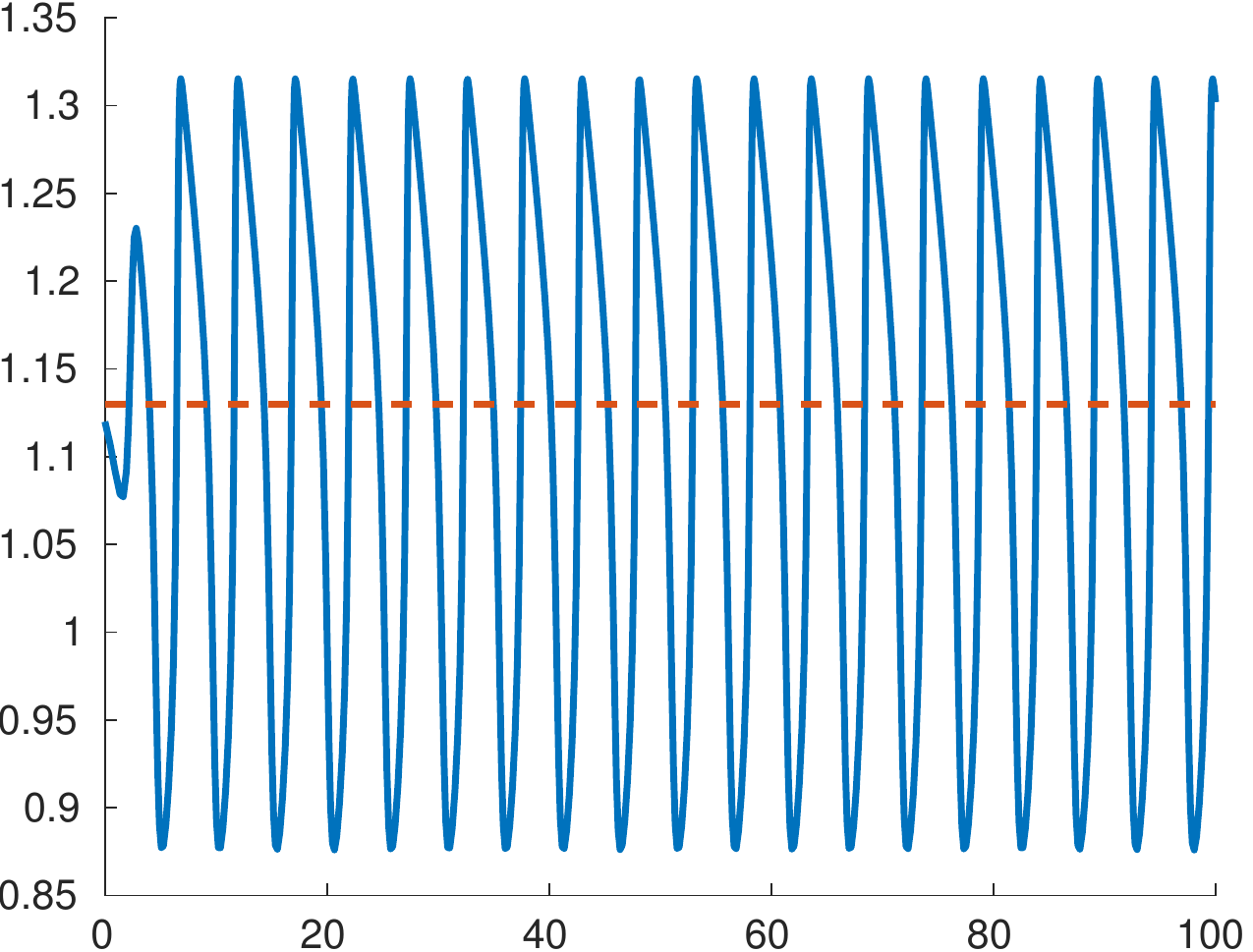}
 \caption{Time dynamics of eggs (left) and larvae (right) for $a_1=0.1$, $b_1^{0.05}=2.91$.}
 \label{fig:eq}
\end{figure}

The parameters $a$ and $b$ are chosen so that Theorem \ref{mainres} applies, which proves the existence of periodic solutions close to the non-trivial steady states.
We perform numerical test by letting a parameter $j$ vary in a set $J$ of $18$ values between $0.05$ and $4$ in order to obtain
$162$ couples $(a_i,b_i^j)_{i=1,\dots,9;j\in J}$ by 
\[
a_{i}=0.1+0.05(i-1) \text{ and } b_i^j=b_{i,{min}}+j\times b_{i,{min}},
\]
where $b_{i,{min}}$ is the minimal $b$ that can be chosen for $a_i$ to obtain oscillations (if $b<b_{i,{min}}$ the solutions can not oscillate), {\it i.e.} for which the trace of the linearized operator is equal to $0$.

The hatching functions are: 
$$h_i^j(L)=a_i\Big(\arctan(b_i^j(L-\Lbar))+\frac{\pi}{2}\Big).$$
In our tests the steady state changes with $i$ (for example $\Ebar_1=145.92$, $\Ebar_4=59.59$
and $\Ebar_9=30$)
but we always have $\Lbar=\Abar\frac{\delta_A}{\tau_L}=1.13$.

\begin{figure}
 \includegraphics[width=.5\linewidth]{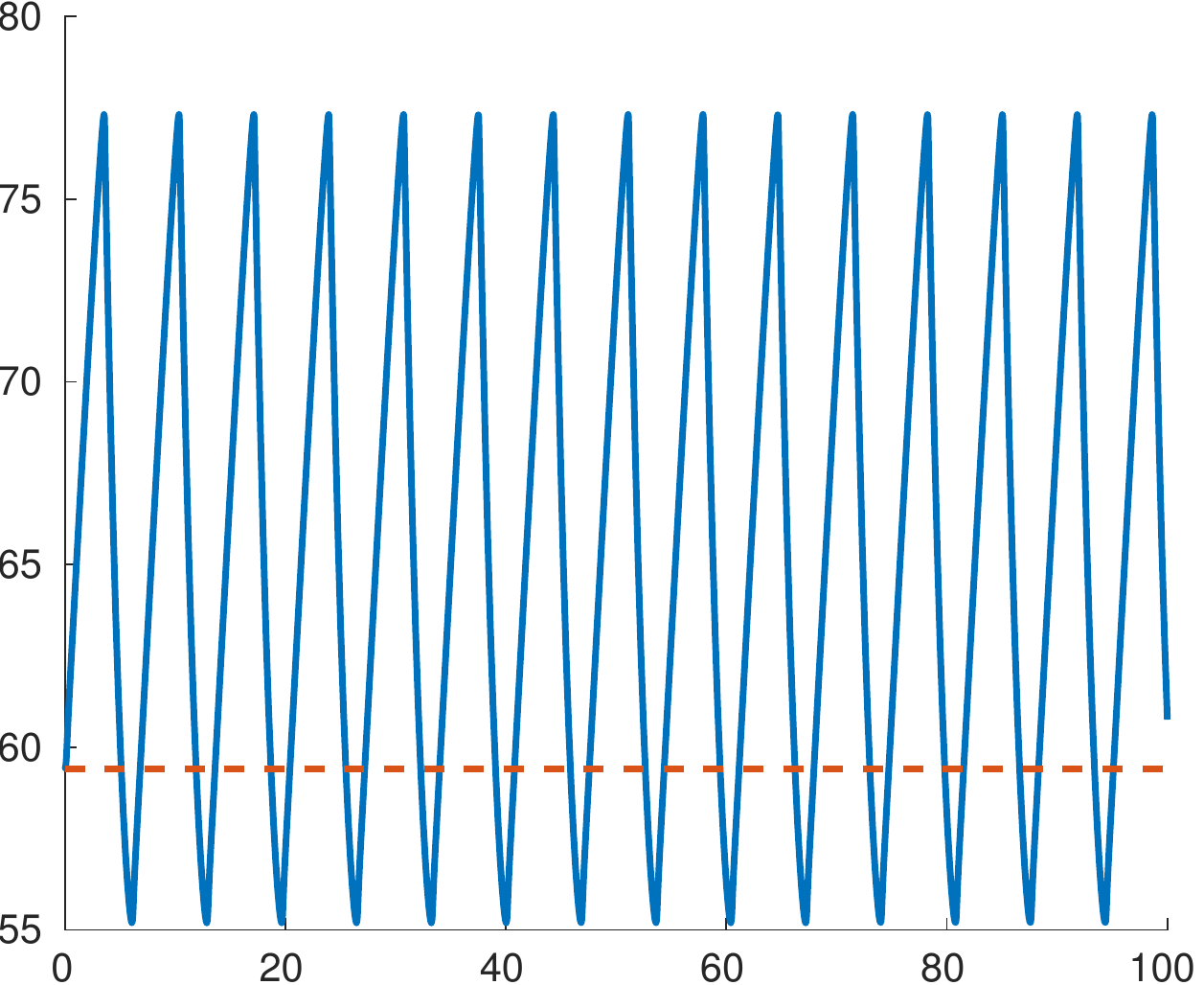}
 \includegraphics[width=.5\linewidth]{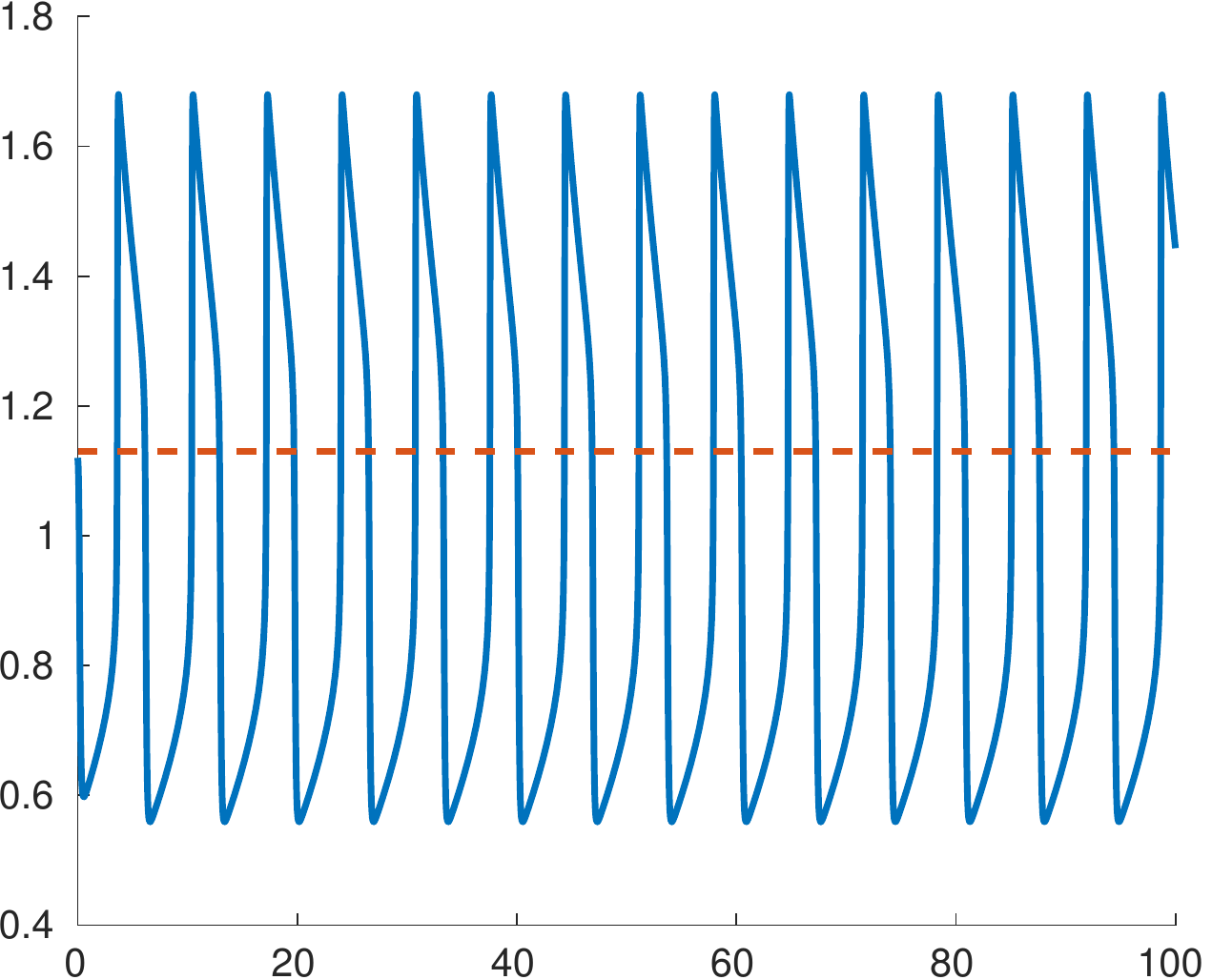}
 \caption{Time dynamics of eggs (left) and larvae (right) for $a_2=0.25$, $b_2^{0.5}=4.18$.}
 \label{fig:eq5}
\end{figure}

\begin{figure}
 \includegraphics[width=.5\linewidth]{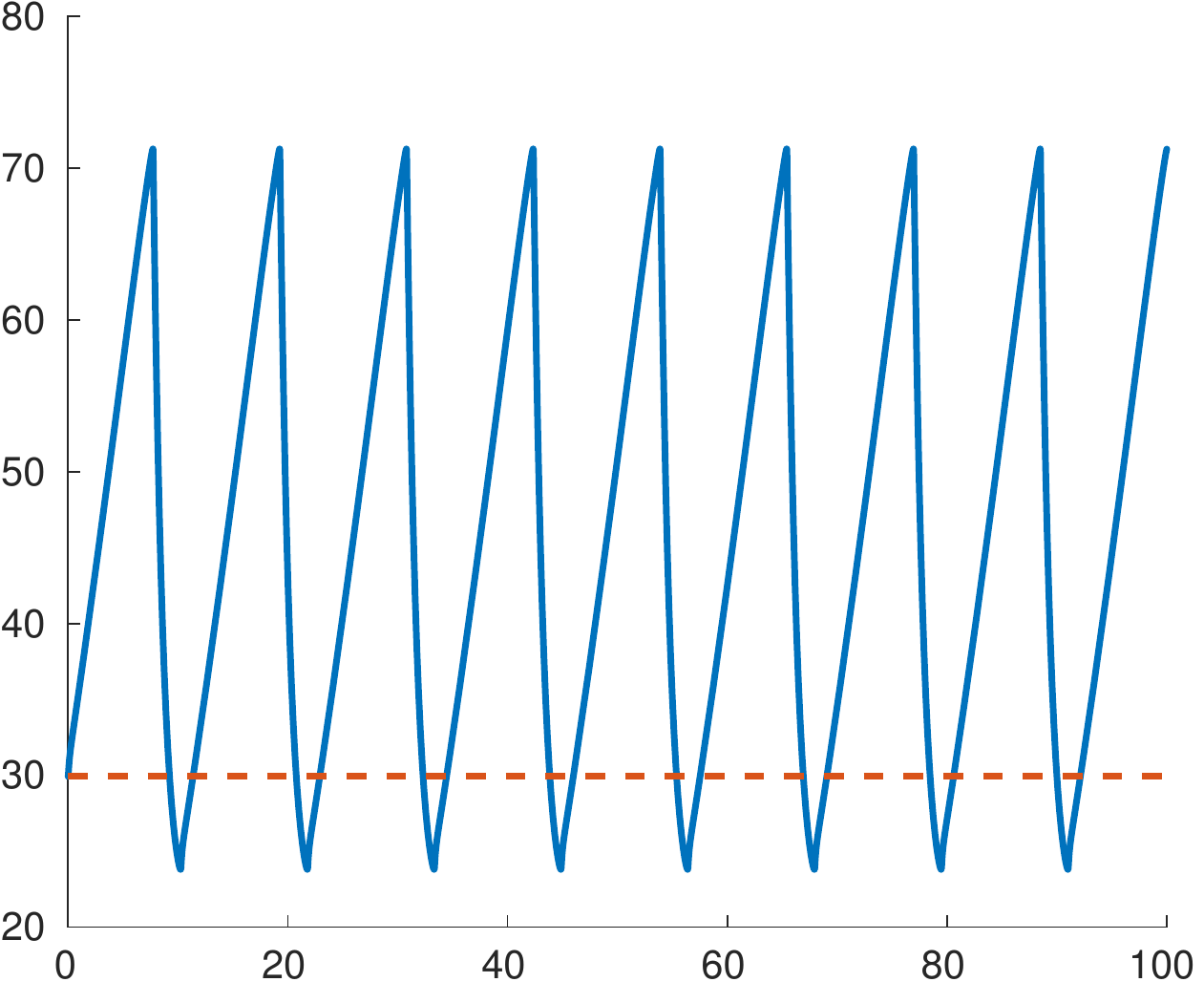}
 \includegraphics[width=.5\linewidth]{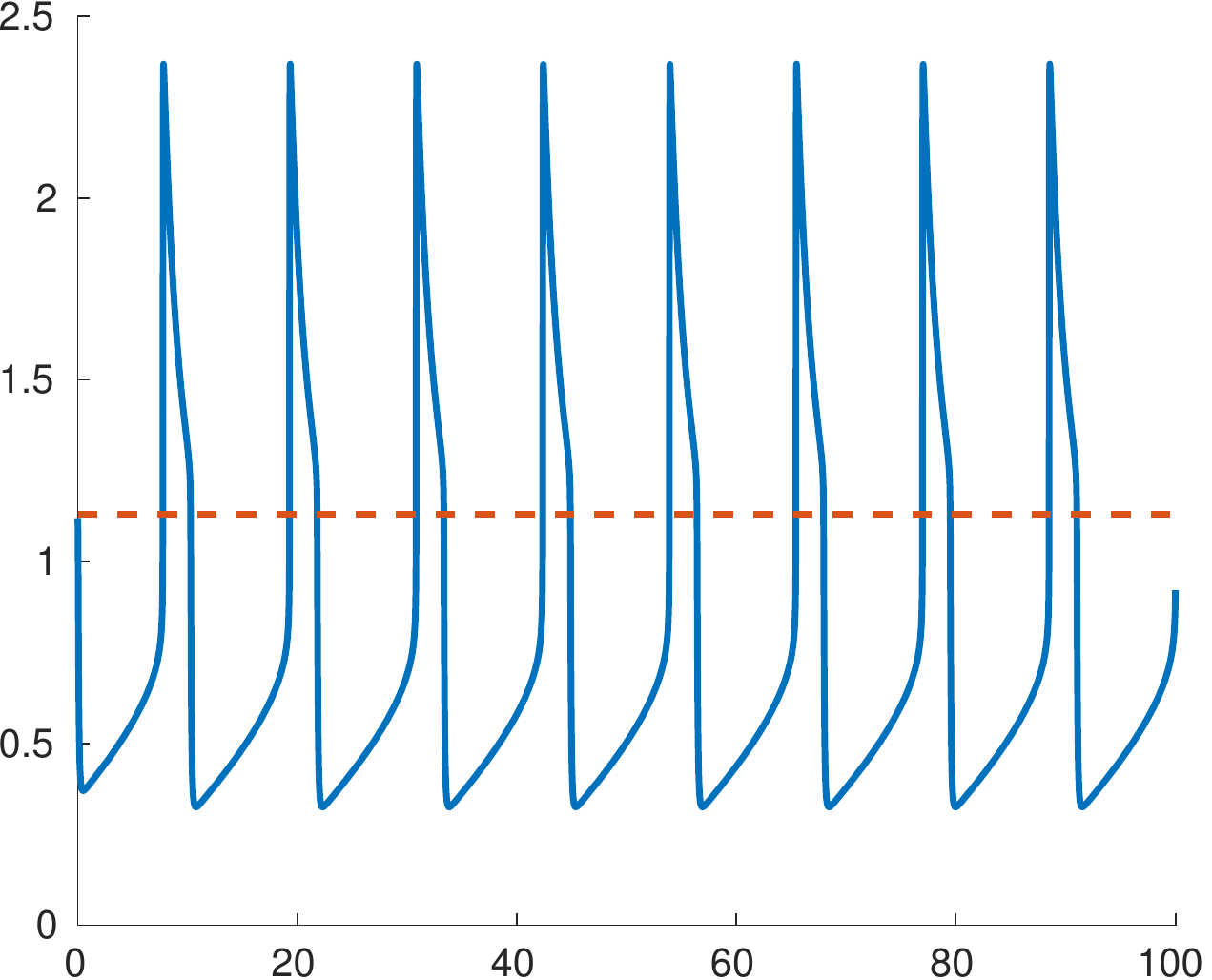}
 \caption{Time dynamics of eggs (left) and larvae (right) for $a_3=0.5$, $b_3^{2}=8.44$.}
 \label{fig:eq9}
\end{figure}

\begin{table}[h]
\centering
\begin{tabular}{|l||c|c|c|c|}
\hline
$a=0.1$ & Period (days) & $\Ebar$ & $\Lbar$ &Larvae amplitude ($\%\Lbar$)\\
\hline
$b=2.91$ & $5.18$ & &&$23.1$\\
\cline{1-2}\cline{5-5}
$b=4.16$ & $15.06$ & $145.92$ &$1.13$&$51.3$\\
\cline{1-2}\cline{5-5}
$b=8.32$ & $61.54$ & &&$110.22$\\
\cline{1-2}\cline{5-5}
{$b=3.52$} & $9.6$ & &&$42.08$\\
\hline
\end{tabular}
\caption{Steady states, period and amplitude of oscillations for $a=.1$}
\label{Tablea=0.1}
\end{table}

\begin{table}[h]
\centering
\begin{tabular}{|l||c|c|c|c|}
\hline
$a=0.25$ & Period (days) & $\Ebar$ & $\Lbar$ &Larvae amplitude ($\%\Lbar$)\\
\hline
$b=2.93$& $2.68$ & && $20.35$\\
\cline{1-2}\cline{5-5}
$b=4.18$ & $6.96$ & $59.59$ & $1.13$ & $50.67$\\
\cline{1-2}\cline{5-5}
$b=8.37$ & $22.64$ & && $110.2$ \\
\cline{1-2}\cline{5-5}
$b=4.99$ & $9.98$ & && $63.03$\\
\hline
\end{tabular}
\caption{Steady states, period and amplitude of oscillations for $a=.25$}
\label{Tablea=0.25}
\end{table}

\begin{table}[h!]
\centering
\begin{tabular}{|l||c|c|c|c|}
\hline
$a=0.5$ & Period (days) & $\Ebar$ & $\Lbar$ &Larvae amplitude ($\%\Lbar$)\\
\hline
$b=2.96$ & $1.72$ & && $18.03$\\
\cline{1-2}\cline{5-5}
$b=4.22$ & $3.8$ & $30$ & $1.13$ & $49.8$\\
\cline{1-2}\cline{5-5}
$b=8.44$ & $11.5$ & && $109.9$\\
\cline{1-2}\cline{5-5}
$b=7.6$ & $10.1$ & && $99.43$\\
\hline
\end{tabular}
\caption{Steady states, period and amplitude of oscillations for $a=.5$}
\label{Tablea=0.5}
\end{table}

\begin{figure}[h!]
\includegraphics[width=.5\linewidth]{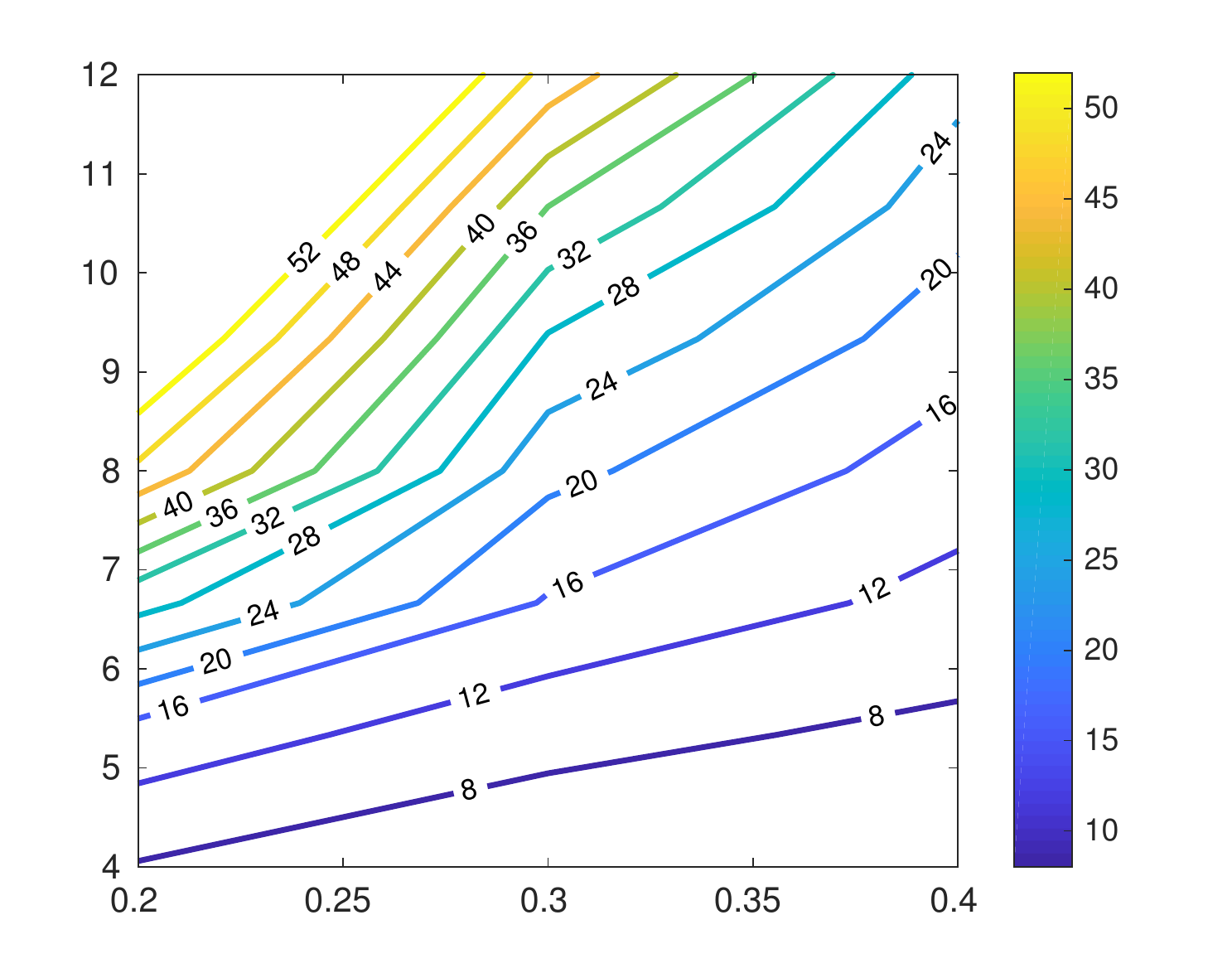}
\includegraphics[width=.5\linewidth]{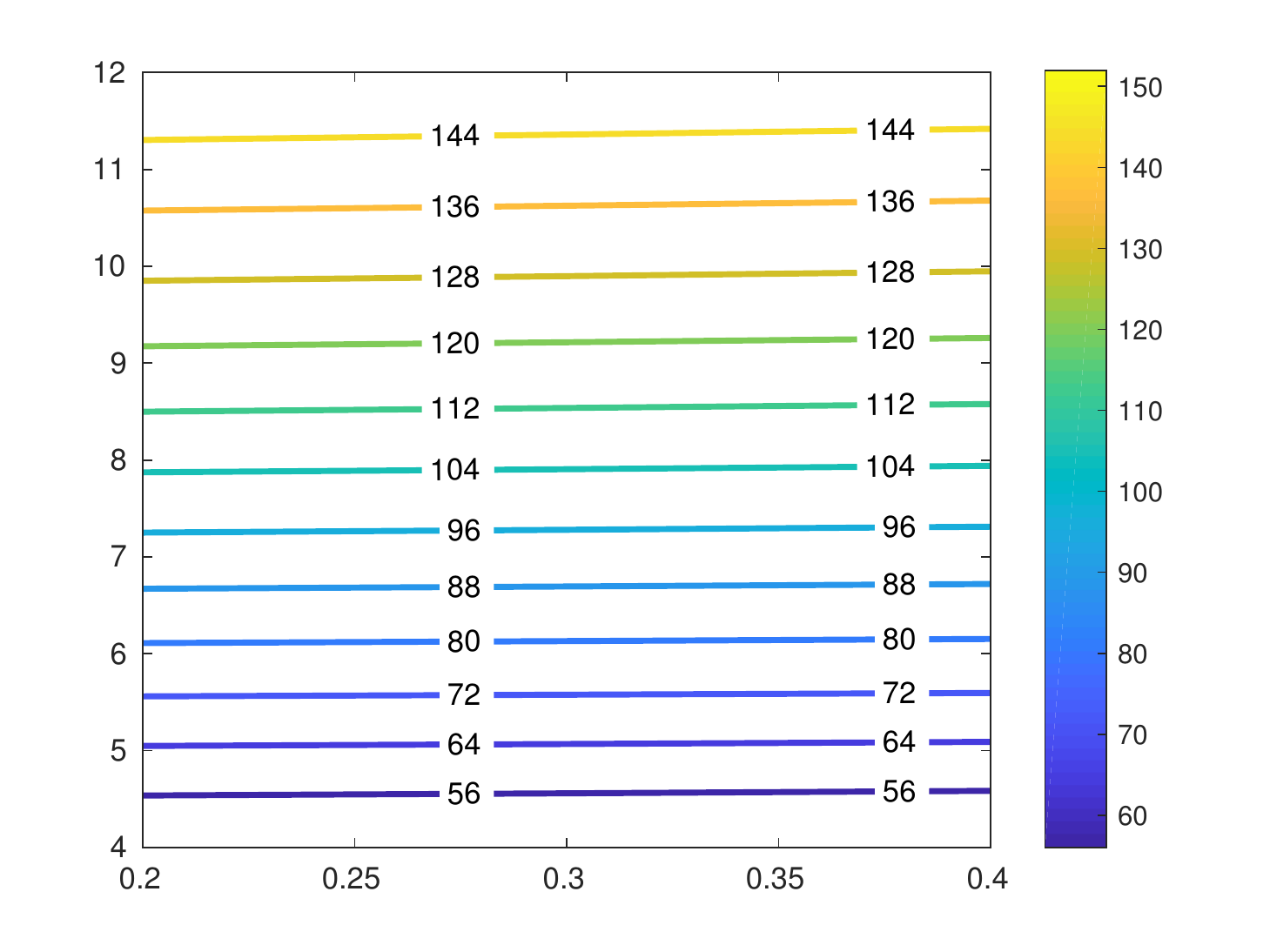}
 \caption{Larvae dynamics period $T_0$ in days ({\it left}) and larvae dynamics amplitude (Amp) in percentage of $\Lbar$ ({\it right}), for different couples $(a,b)$. }
 \label{fig3}
\end{figure}

We provide numerical results for $i \in \{1, 4, 9\}$ and $j \in \{ 0.1, 0.25, 0.5\}$ initial data close to the steady state (which is drawn in dashed line). Two sets of initial data are chosen, $(E(0),L(0))=(\Ebar,\Lbar)$ (green) and $(E(0),L(0))=(\Ebar,\Lbar+0.02)$ (blue), which gives oscillations that appear to be periodic in time.
Simulations are made with $a=a_1$ in Figure $\ref{fig:eq}$ (with $b=b_1^{0.05}$) and a time variable evaluated in $[0,100]$ days ; $a=a_2$ in Figure $\ref{fig:eq5}$ (with $b=b_2^{0.5}$) and a time variable evaluated in $[0,100]$ days ; $a=a_3$ in Figure $\ref{fig:eq9}$ (with $b=b_3^2$) and a time variable evaluated in $[0,150]$ days.

Considering the blue curves, we sum up in the Tables \ref{Tablea=0.1}, \ref{Tablea=0.25} and \ref{Tablea=0.5} what we obtain for the period and the oscillations' amplitude taken by the solutions. In the last line of the tables we give a value of $b$ that can be chosen to obtain a period of about 10 days.
Relative amplitude of the oscillations is expressed as a percentage of the (constant) value $\Lbar$.

It is possible to achieve the same period $T_0$ for different couples of parameters $(a,b)$. 
For a fixed $a$, when $b$ is increasing, the period $T_0$ and the amplitude of larvae are increasing too. The amplitude, on the contrary, mainly depends on $b$. This is illustrated in Figure $\ref{fig3}$.

\bibliographystyle{abbrv}
\bibliography{biblio}

\begin{thebibliography}{10}

\bibitem{maple}
{\em Maple 18.}
\newblock Maplesoft, a division of Waterloo Maple Inc., Waterloo, Ontario.

\bibitem{Azn.Model}
V.~R. Aznar, M.~S.~D. Majo, S.~Fischer, D.~Francisco, M.~A. Natiello, and H.~G.
  Solari.
\newblock {A model for the development of {\it Aedes} (Stegomyia) {\it aegypti}
  as a function of the available food}.
\newblock {\em Journal of Theoretical Biology}, 365:311 -- 324, 2015.

\bibitem{Azn.Modeling}
V.~R. Aznar, M.~Otero, M.~S.~D. Majo, S.~Fischer, and H.~G. Solari.
\newblock {Modeling the complex hatching and development of {\it Aedes aegypti}
  in temperate climates}.
\newblock {\em Ecological Modelling}, 253:44 -- 55, 2013.

\bibitem{Bar.Effect}
J.~Bara, Z.~Rapti, C.~E. Cáceres, and E.~J. Muturi.
\newblock {Effect of larval competition on extrinsic incubation period and
  vectorial capacity of {\it Aedes albopictus} for dengue virus}.
\newblock {\em PLoS ONE}, 10(5):1--18, 2015.

\bibitem{Bha.Burden}
S.~Bhatt, P.~W. Gething, O.~J. Brady, J.~P. Messina, A.~W. Farlow, C.~L. Moyes,
  J.~M. Drake, J.~S. Brownstein, A.~G. Hoen, O.~Sankoh, M.~F. Myers, D.~B.
  George, T.~Jaenisch, G.~R.~W. Wint, C.~P. Simmons, T.~W. Scott, J.~J. Farrar,
  and S.~I. Hay.
\newblock {The global distribution and burden of dengue}.
\newblock {\em Nature}, 496(7446):504--507, 2013.

\bibitem{DHM}
O.~Diekmann, J.~Heesterbeek, and J.~Metz.
\newblock {On the definition and the computation of the basic reproduction
  ratio $R_0$ in models for infectious diseases in heterogeneous populations}.
\newblock {\em J. Math. Biol.}, 28:365--382, 1990.

\bibitem{Dut.Lab}
G.~L.~C. Dutra, L.~M.~B. dos Santos, E.~P. Caragata, J.~B.~L. Silva, D.~A.~M.
  Villela, R.~Maciel-de Freitas, and L.~Andrade~Moreira.
\newblock {From Lab to Field: the influence of urban landscapes on the invasive
  potential of \textit{Wolbachia} in Brazilian \textit{Aedes aegypti}
  mosquitoes}.
\newblock {\em PLoS Neglect Trop D}, 9(4), 2015.

\bibitem{art:hatch}
J.~Edgerly and M.~Marvier.
\newblock {To hatch or not to hatch? Egg hatch response to larval density and
  to larval contact in a treehole mosquito}.
\newblock {\em Ecological entomology}, 17:28--32, 1992.

\bibitem{Ermentrout}
B.~Ermentrout.
\newblock {\em Simulating, Analyzing, and Animating Dynamical Systems}.
\newblock Society for Industrial and Applied Mathematics, 2002.

\bibitem{JP}
J.-P. Fran\c{c}oise.
\newblock {\em {Oscillations en biologie, Analyse qualitative et modèle}}.
\newblock Springer, 2005.

\bibitem{Guz.Potential}
G.~Guzzetta, F.~Montarsi, F.~A. Baldacchino, M.~Metz, G.~Capelli, A.~Rizzoli,
  A.~Pugliese, R.~Rosà, P.~Poletti, and S.~Merler.
\newblock {Potential risk of dengue and chikungunya outbreaks in northern Italy
  based on a population model of {\it Aedes albopictus} (Diptera: Culicidae)}.
\newblock {\em PLoS Neglect Trop D}, 10(6):1--21, 06 2016.

\bibitem{Hof.Stability}
A.~A. Hoffmann, I.~Iturbe-Ormaetxe, A.~G. Callahan, B.~L. Phillips,
  K.~Billington, J.~K. Axford, B.~Montgomery, A.~P. Turley, and S.~L. O'Neill.
\newblock {Stability of the \textit{w}Mel \textit{Wolbachia} infection
  following invasion into \textit{Aedes aegypti} populations}.
\newblock {\em PLoS Neglect Trop D}, 8(9):1--9, 09 2014.

\bibitem{art:osc}
N.~Honorio, C.~Codeço, F.~Alves, M.~Magalh\~{a}es, and R.~Lourenço-de
  Oliveira.
\newblock {Temporal distribution of {\it Aedes aegypti} in different districts
  of Rio De Janeiro, Brazil, measured by two types of traps}.
\newblock {\em J Med Entomo}, 46 (5):1001--1014, 2009.

\bibitem{Jul.She}
S.~Juliano, R.~G.S., R.~Maciel-de Freitas, M.~Castro, C.~Code\c{c}o,
  R.~Louren\c{c}o-de Oliveira, and L.~Lounibos.
\newblock {She{\rq}s a femme fatale: low-density larval development produces
  good disease vectors}.
\newblock {\em Mem{\'o}rias do Instituto Oswaldo Cruz}, 109(8):1070--1077, dec
  2014.

\bibitem{AWD}
J.~Koiller, M.~A. da~Silva, M.~O. Souza, C.~T. Codeço, A.~Iggidr, and
  G.~Sallet.
\newblock {Aedes, Wolbachia and dengue}.
\newblock {\em Project-Team MASAIE}, 2014.

\bibitem{Leg.Comparison}
M.~Legros, M.~Otero, V.~Romeo~Aznar, H.~Solari, F.~Gould, and A.~L. Lloyd.
\newblock {Comparison of two detailed models of {\it Aedes aegypti} population
  dynamics}.
\newblock {\em Ecosphere}, 7(10), 2016.
\newblock e01515.

\bibitem{Liv.Complex}
T.~P. Livdahl, R.~K. Koenekoop, and S.~G. Futterweit.
\newblock {The complex hatching response of {\it Aedes} eggs to larval
  density}.
\newblock {\em Ecological Entomology}, 9(4):437--442, 1984.

\bibitem{cracken}
J.~Marsden and M.~McCracken.
\newblock {\em {The Hopf Bifurcation and its Applications}}, volume~19 of {\em
  Applied mathematical sciences}.
\newblock Springer-Verlag, 1976.

\bibitem{forme}
J.~Meiss.
\newblock {\em {Differential Dynamical Systems}}.
\newblock SIAM, 2007.

\bibitem{Murray}
J.~D. Murray.
\newblock {\em Mathematical biology. {I}. An introduction}.
\newblock Interdisciplinary applied mathematics. Springer, New York, 2002.

\bibitem{Perthame}
B.~Perthame.
\newblock {\em {Parabolic equations in biology}}.
\newblock Lecture Notes on Mathematical Modelling in the Life Sciences.
  Springer International Publishing, 2015.

\bibitem{VdDW}
P.~van~den Driessche and J.~Watmough.
\newblock {A simple SIS epidemic model with a backward bifurcation}.
\newblock {\em J Math Biol}, 40:525--540, 2000.

\bibitem{Density}
D.~Villela, C.~Codeço, F.~Figueiredo, G.~Garcia, R.~Maciel-de Freitas, and
  C.~Struchiner.
\newblock {A Bayesian hierarchical model for estimation of abundance and
  spatial density of {\it Aedes aegypti}}.
\newblock {\em PLoS ONE}, 10(4), 2015.
\newblock e0123794.

\bibitem{Yan.Assessing}
H.~Yang.
\newblock {Assessing the influence of quiescence eggs on the dynamics of
  mosquito {\it Aedes aegypti}}.
\newblock {\em Applied Mathematics}, 5:2696--2711, 2014.

\bibitem{Temp}
H.~Yang, M.~Macoris, K.~Galvani, M.~Andrighetti, and D.~Wanderley.
\newblock {Assessing the effects of temperature on the population of {\it Aedes
  aegypti}, the vector of dengue}.
\newblock {\em Epidemiol Infect}, 137:1188--1202, 2009.

\bibitem{Yea.Mitochondrial}
H.~L. Yeap, G.~Rasic, N.~M. Endersby-Harshman, S.~F. Lee, E.~Arguni, H.~L.
  Nguyen, and A.~A. Hoffmann.
\newblock {Mitochondrial DNA variants help monitor the dynamics of {\it
  Wolbachia} invasion into host populations}.
\newblock {\em Heredity}, 116(3):265--276, 2016.

\end{thebibliography}

\end{document}